\documentclass[11pt,reqno]{amsart}

\usepackage[a4paper]{geometry}
 \usepackage[utf8]{inputenc}
\usepackage[T1]{fontenc}
\usepackage{lmodern}
\usepackage{amsmath,amssymb}
\usepackage{amsthm}
\usepackage{paralist}
\usepackage{float}
\usepackage{accents}
\usepackage{color}
\usepackage[authoryear]{natbib}
\usepackage{bm}
\usepackage{lscape}
\usepackage{subfig}
\usepackage{graphicx}
\usepackage{booktabs}
\numberwithin{equation}{section}
\usepackage{dsfont}
\usepackage[colorlinks=true,linkcolor=blue,citecolor=blue,pdfborder={0 0 0}]{hyperref}

\newcommand{\id}{\mathds{1}}

\renewcommand{\epsilon}{\varepsilon}
\newcommand{\eps}{{\varepsilon}}
\renewcommand{\phi}{\varphi}

\newcommand{\R}{\mathbb{R}}
\newcommand{\Z}{\mathbb{Z}}
\newcommand{\N}{\mathbb{N}}

\newcommand{\pr}{\mathbb{P}}        
\newcommand{\ex}{\mathbb{E}}        

\newcommand{\var}{\textnormal{Var}} 
\newcommand{\cov}{\textnormal{Cov}} 

\newcommand{\MSE}{\textnormal{MSE}}

\newcommand{\vertiii}[1]{{\left\vert\kern-0.25ex\left\vert\kern-0.25ex\left\vert #1
    \right\vert\kern-0.25ex\right\vert\kern-0.25ex\right\vert}}

\newcommand{\Ac}{\mathcal{A}}

\newcommand{\Dc}{\mathcal{D}}

\newcommand{\Fc}{\mathcal{F}}

\newcommand{\Lc}{\mathcal{L}}

\newcommand{\Nc}{\mathcal{N}}
\newcommand{\Oc}{\mathcal{O}}

\newcommand{\argmin}{\textnormal{argmin}}

\newcommand{\diff}{{\mathrm{d}}}
\newcommand{\supp}{\textnormal{supp}}

\newcommand{\scs}{\scriptscriptstyle}

\newcommand{\convd}{\rightsquigarrow}              
\newcommand{\convw}{\convd}                           

\newtheorem{theorem}{Theorem}[section]

\newtheorem{lemma}[theorem]{Lemma}
\newtheorem{corollary}[theorem]{Corollary}

\theoremstyle{definition}

\newtheorem{alg}[theorem]{Algorithm}
\newtheorem{assumption}[theorem]{Assumption}

\newtheorem{remark}[theorem]{Remark}

\definecolor{gray}{gray}{0.6}

\allowdisplaybreaks[4]

\begin{document}

\parindent 0cm

\title[Structural breaks in locally stationary processes]{
A distribution free test for changes  in  the trend function of locally stationary processes}

\author{Florian Heinrichs}
\author{Holger Dette}
\address{Ruhr-Universit\"at Bochum, Fakult\"at f\"ur Mathematik, Universit\"atsstr.\ 150, 44780 Bochum, Germany.}
\email{florian.heinrichs@rub.de}
\email{holger.dette@rub.de}

\begin{abstract}

In the common time series  model  $X_{i,n}  = \mu (i/n) + \varepsilon_{i,n}$  with    non-stationary errors we  consider the problem of 
detecting a significant    deviation  of  the mean function $\mu$  from a benchmark $g (\mu )$  (such as the initial value $\mu (0)$ or the  average trend  $\int_{0}^{1} \mu (t) dt$).
The problem is motivated by a more realistic modelling of change point analysis, where one is interested in identifying relevant deviations in  a smoothly varying  sequence of means $  (\mu (i/n))_{i =1,\ldots ,n }$
and cannot assume that the sequence is piecewise constant.
A test for this type of hypotheses is developed  using an appropriate estimator for  the 
integrated squared deviation of   the mean function and the threshold.
By a new concept of self-normalization adapted to non-stationary processes an asymptotically  pivotal test for the hypothesis of a relevant deviation is constructed.
The results are illustrated by means of a simulation study and a  data example.
			
%

\medskip


\end{abstract}

\date{\today}

\maketitle


\section{Introduction} \label{sec:intro}

Within the last decades, the detection of structural breaks in time series has become a very active area of research with many
 applications in fields like climatology, economics, engineering, genomics, hydrology, etc. \citep[see][among  many others]{auehor2013,Jandhyala2013,WoodallMontgomery2014,Sharma2016,ChakrabortiGraham2019,truong2020}. In the simplest case, one is interested in detecting structural breaks in the sequence of means $(\mu_i)_{i=1,\dots,n}=(\mu(i/n))_{n \in \mathbb{N}}$ of a time series $(X_{i,n})_{i=1,\dots,n}$
 corresponding to a location model of the form
\begin{equation}\label{eq:model}
X_{i,n}=\mu(i/n)+\eps_{i,n}~,~~i=1, \ldots , n~.
\end{equation}
A large part of  the literature considers the problem of detecting changes in a piecewise constant
mean function $\mu:[0,1]\to \R$, where early references assume the existence  of
at most one change point \citep[see, e.\,g.][among others]{Priestley1969, Wolfe1984, Horvath1999} and more recent literature
investigates  multiple change points \citep[see, e.g.][among  many others]{frick2014,fryzlewicz2018a,detschvet2018,baranowski2019}.
The   errors $(\eps_{i,n})_{i=1,\dots,n}$  in model \eqref{eq:model} are usually assumed to form at least  a stationary process and many theoretical results
for detecting multiple change points  are only available for independent identically distributed error processes.
These assumptions  simplify the statistical analysis of structural breaks substantially, as - after removing the piecewise constant
trend - one can
work under the assumption of a stationary or an  independent identically distributed error process
and smoothing is not necessary to estimate the trend function.

On the other hand, the assumption of a strictly piecewise constant mean function might not be realistic in
  many  situations and  it might be  more reasonable to assume that $\mu$ varies smoothly rather than abrupt.
  A typical example is temperature data \citep[see, e.\,g.][]{karl1995,collins2000}
  where  it might be of more  interest to investigate whether the mean function deviates fundamentally from a given benchmark denoted by  $g(\mu)$.
  Here $g$ is a functional of the mean function, such as the value at the point $0$, that is
  $
  g(\mu) = \mu (0),
  $ or an average over a certain time period, that is
\begin{equation}
  \label{hol3}
  g(\mu) = \frac{1}{t_1-t_0}\int_{t_0}^{t_1} \mu (x) dx
  ~~\text{ for some } ~~0 \leq t_{0} <  t_{1}  \leq 1
\end{equation}
   (see Section \ref{sec2} for more details).
  Moreover,
there also exist  many time series exhibiting  a non-stationary behaviour in the higher order moments and dependence structure   \citep[see][among others]{stuaricua2005nonstationarities,elsner2008increasing,guillaumin2017analysis}, and the
detection  of fundamental deviations from a benchmark in a sequence of gradually changing means
under  the assumption of a location model with a stationary error process  might be misleading.

In this paper we propose a distribution free test for relevant deviations of the mean function $\mu$ from a given benchmark $g(\mu)$ in a
location scale model of the form \eqref{eq:model} with a non-stationary error process. More precisely, for some pre-specified threshold $\Delta > 0$ we are interested in testing the hypotheses
\begin{equation}\label{h0:rel}
H_0: d_0 = \Big (\int_0^1 \big(\mu(x)-g(\mu)\big)^2\tau (\diff x )\Big )^{1/2} \leq \Delta \quad \text{vs.} \quad H_1: d_0>\Delta,
\end{equation}
where $\tau$ is an appropriate measure on the interval $[0,1]$ chosen by the statistician. This means that we are looking for  ``substantial''  deviations of   the  mean function  from a given benchmark  $g(\mu)$
 in an $L^2$-sense. The choice of the threshold depends on the particular application and is related  to a balance between bias and variance
 as the detection of deviations from a (constant) mean often results in an adaptation of the statistical analysis (for example in forecasting).
As such an  analysis is  performed ``locally'',  resulting estimators will have a smaller bias
but a larger variance. However, if the changes in the signal are only weak, such an adaptation might not be necessary
because a potential  decrease in bias might be overcompensated by an increase of variance.

 In principle, a test for the hypotheses in \eqref{h0:rel} could be developed  using a nonparametric estimate of  the mean function $\mu$ to obtain an estimate, say $\hat d_0$, of the distance $d_0$. The null hypothesis in \eqref{h0:rel} is then rejected for large values of  $\hat d_0$. However, the distribution of the test statistic will depend in an intricate way on the dependence structure  of the non-stationary error process  in model \eqref{eq:model}, which is difficult to estimate. To address this problem we will introduce a new  concept of self-normalization  and construct an (asymptotically) pivotal test statistic for the hypotheses in  \eqref{h0:rel}.
 The basic idea  of our approach is to permute the data and consider the partial sum process of this permutation, thus, taking into account observations over the whole interval rather than only the first observations.  The new concept and the asymptotic properties of the standardized statistic can be found in Section \ref{sec3}, while some details on the testing problem and
 mathematical background on locally stationary processes are  introduced in Section \ref{sec2}.
In  Section \ref{sec4} we  investigate the  finite sample properties of the proposed testing procedure by means of a simulation study and provide an application to temperature data.
Finally, in Section \ref{secA}, the proofs of  the theoretical results in Section \ref{sec3} are presented.

 \subsection{Related literature} Despite of its importance the problem of detecting relevant deviations in a sequence of gradually changing means has only been considered by a few authors.  \cite{dettewu2019} investigate a mass excess approach for this problem. More precisely, these authors measure deviations from the benchmark by
 the Lebesgue measure of the set $\{t\in[0,1]:|\mu(t)-g(\mu)|>\Delta\}$ and test whether this quantity exceeds a certain threshold $c>0$.  Their approach requires estimation of the
 local long-run variance and multiplier bootstrap.
 More recently,   \cite{bucher2020} propose   the maximal distance to measure relevant  deviations from the benchmark and consider the null hypothesis
 $H_0: \sup_{t\in[0,1]}|\mu(t)-g(\mu)| \leq \Delta$. While the maximum deviation might be easy to interpret for practitioners,
 the asymptotic analysis of a corresponding estimate is challenging. In particular  it requires an estimation of the long-run variance and additionally the estimation
 of the sets, where the absolute difference $|\mu(t)-g(\mu)| $ attains its sup-norm. The methodology proposed here avoids the problem of estimating tuning parameters of this
 type using an $L^2$-norm in combination with a new concept of self-normalization.

Ratio statistics or self-normalization have  been  introduced by \cite{horvath2008} and
 \cite{Shao2010b} in the context of change point detection in stationary processes and
avoid a direct estimation of the long-run variance through a convenient rescaling of the test statistic.
The currently available self-normalization procedures are based on  partial sum processes  \citep[see][for a recent review]{Shao2015},
which usually (under the assumption of stationarity) have a limiting process of the form $\{ \sigma W(\lambda) \}_{\lambda \in [0,1]}$,
where $\{ W(\lambda)\}_{\lambda\in[0,1]}$ is a known stochastic process and $\sigma$  an unknown factor encapsulating the
dependency structure of the underlying process. In this case the factorisation of the limit into the long-run variance and a probabilistic term is used to construct a
pivotal test statistic by forming a ratio such that the factor $\sigma$  in the numerator and denominator cancels. However,
in the case of  non-stationarity, the situation is more  complicated, because the limiting process is of the
form $\big \{  \int_0^\lambda \sigma(u)\diff W(u) \big\}_{\lambda \in [0,1]}$ such that the probabilistic
 and the part representing the dependence structure cannot be separated.
 \cite{zhao2013} and \cite{rho2015} discuss in fact these problems in the context of locally stationary time series, but the proposed self-normalizations need to be combined with a wild bootstrap.
 In this paper, we present a full self-normalization procedure for non-stationary time series, which might be also useful for testing classical hypotheses.

	\section{The testing problems and mathematical Preliminaries} \label{sec2}

Throughout  this paper  $\Lc^2([0,1])$ denotes the space of real-valued square-integrable functions on $[0,1]$ and $L^2([0,1])$ the corresponding normed vector space of equivalence classes. Let $\langle f,g\rangle = \int_0^1 f(x)g(x)\diff x$ denote the scalar product in $L^2([0,1])$ and $\|f\|_2 = \langle f,f\rangle^{1/2}$ the corresponding norm, for $f,g\in L^2([0,1])$. Further, $\langle f,g\rangle_\tau = \int_0^1 f(x)g(x)\tau(\diff x)$ and $\|f\|_{2,\tau}=\langle f,f\rangle_\tau^{1/2}$, for $f,g\in L^2([0,1],\tau)$. Finally, for the sake of readability, for functions in $\Lc^2([0,1])$, we denote the integral $\int_0^1 f(x)g(x)\diff x$ by $\langle f,g\rangle$. 
Finally, if $X$ is  a real-valued random variable we use the notation (in the case of existence)  $\|X\|_{q,\Omega} = \big(\ex[|X|^q]\big)^{1/q}$, for $q\geq 1$.

\subsection{Relevant deviations in a sequence of gradually changing means} \label{sec21}

Recall the definition of model  \eqref{eq:model} and the hypotheses \eqref{h0:rel}. Different
benchmarks may be of interest in applications.
For example, if one  is interested in deviations from the value of the mean function at a given  time, say  $t \in [0,1]$, one could choose $g(\mu) = \mu (t)$,
while relevant deviations from an average over a certain time period are obtained for the choice \eqref{hol3}.
In  particular if   $t_{0}$,  $t_{1}$ and  $\tau$ 
are chosen $0$, $1$ and 
the Lebesgue measure, respectively,  one compares
 the local mean $\mu(x)$ with the overall mean $ g(\mu ) =   \bar{\mu} = \int_0^1 \mu(y)\diff y$  and the hypotheses
 in \eqref{h0:rel} read as follows
$$ H_0:
\Big ( \int_0^1 \big(\mu(x)-\bar{\mu}\big)^2 \diff x \Big )^{1/2}\leq \Delta~\quad \text{vs.}\quad  \Big (\int_0^1 \big(\mu(x)-\bar{\mu}\big)^2 \diff x \Big )^{1/2}> \Delta.
$$
The tests which will be developed in this paper are based on an appropriate estimate of the quantity
\begin{equation}
  \label{hol4}
d_0 = \Big (\int_0^1 \big(\mu(x)-g(\mu)\big)^2 \tau (\diff x ) \Big )^{1/2}
\end{equation}
for which we require precise estimates of the mean  function $\mu$ and  the threshold $g(\mu)$. Note that the measure $\tau$ in \eqref{hol4} is chosen by the statistician and therefore known.

Throughout this paper, we assume that   $\tau$ is absolutely continuous with respect to the Lebesgue measure and   has  a piecewise continuous density, say $f_\tau$.
Further,  we assume that the mean function $\mu$ is sufficiently smooth, as specified in the following assumption.

\begin{assumption}\label{assumption:C2}
	The function $\mu:[0,1]\to\R$ is twice differentiable with Lipschitz continuous second derivative. {In particular,  this implies that the integrals $\int_0^1 \mu^2(x)\diff x$ and $\int_0^1 \mu^2(x) \tau (\diff x )$ are finite, thus, $\mu\in \Lc^2([0,1])$ and $\mu\in \Lc^2([0,1],\tau)$.}
\end{assumption}

A natural idea for the construction of a test of the hypotheses \eqref{h0:rel} is
 to estimate the $L^2$-distance $d_0$ as defined in \eqref{hol4} and to reject the null hypothesis for large values of the corresponding estimate.
 For this purpose one can use   the \textit{local linear estimator}, which is defined as the first coordinate of
the vector
\begin{align} \label{hol8}
\big(\hat{\mu}_{h_n}(t),\widehat{\mu'}_{h_n}(t)\big) = \underset{b_0,b_1}{\argmin} \sum_{i=1}^{n} \big(X_{i,n}-b_0-b_1(i/n-t)\big)^2 K_{h_n}(i/n-t),
\end{align}
to estimate  the mean function $\mu$ locally
 \cite[see, for example][]{FanGij1996}. In order to reduce the bias we
consider the \textit{Jackknife estimator}
\begin{align} \label{hol8a}
\check{\mu}_{h_n}(t)=2 \hat{\mu}_{h_n/\sqrt{2}}(t)-\hat{\mu}_{h_n}(t)
\end{align}
 as proposed by \cite{schucany1977}
   and obtain an estimate $\check g_{n} = g ( \hat{\mu}_{h_n})$ of  the threshold $g(\mu )$ (other estimates could be used as well).
Here $h_n$ is a  positive bandwidth  satisfying $h_n=o(1)$ as $n\to\infty$,   $K_h(\cdot)=K(\frac{\cdot}{h})$ and
$K$ denotes a kernel function satisfying the following assumption.

\begin{assumption}\label{assump:kern}
	The kernel $K$ is non-negative, symmetric, supported on the interval $[-1, 1]$. It is twice differentiable,  satisfies  $\int_{[-1,1]}K(x) dx = 1$  and Lipschitz continuous in an open interval containing the interval $[-1,1]$.
\end{assumption}

\noindent
The estimate of $d_{0}$ can then be defined as
\begin{equation} \label{hol5}
\|\check{\mu}_{h_n}-\check{g}_n\|_{2,\tau} =  \Big (\int_0^1 \big(  \check{\mu}_{h_n} (x)- \check g_{n} \big)^2 \tau (\diff x )\Big )^{1/2} ~.
\end{equation}
To study the  asymptotic properties of the  statistic defined in \eqref{hol5} and alternative estimates proposed in this paper (see  Section \ref{sec3}
for more details)
we require several assumptions regarding the dependency structure of the error process in model
\eqref{eq:model}, which will be discussed next.

\subsection{Locally stationary processes} \label{sec22}

For the proofs of our main results we require several assumption on the dependence structure  of the non-stationary time series defined in \eqref{eq:model}.
In
the following, we work with the notion of local stationarity as introduced by \cite{Zhou2009}. To be precise, let $\eta=(\eta_i)_{i\in\Z}$ be a sequence of independent identically distributed random variables and let $(\eta')=(\eta_i')_{i\in\Z}$ be an independent copy of $\eta$. Further, define $\Fc_i=\{ \eta_k:k\leq i \}$ and $\Fc_i^*=(\ldots,\eta_{-2},\eta_{-1},\eta_0',\eta_1,\ldots,\eta_i)$. Let $G:[0,1]\times \R^\infty \to \R$ denote a filter, such that $G(t,\Fc_i)$ is a properly defined random variable for all $t\in[0,1]$.

\smallskip


A triangular array $\{(\eps_{i,n})_{1\leq i\leq n}\}_{n\in\N}$ is called \textit{locally stationary}, if there exists a filter $G$, which is continuous in its first argument, such that 
$$
\eps_{i,n}=G(i/n,\Fc_i)
$$  for all $i\in\{1,\ldots,n\}, n\in\N$.
		The {\it  physical dependence measure  of a filter $G$}  with $\sup_{t\in[0,1]}\|G(t,\Fc_i)\|_{q,\Omega}<\infty$ with respect to $\|\cdot\|_{q,\Omega}$ is defined by
		\[ \delta_q(G,i)=\sup_{t\in[0,1]}\|G(t,\Fc_i)-G(t,\Fc_i^*)\|_{q,\Omega}. \]
		A filter $G$ is called {\it  Lipschitz continuous with respect to $\|\cdot\|_{q,\Omega}$}, if
		\[ \sup_{0\leq s<t\leq 1}\|G(t,\Fc_i)-G(s,\Fc_i)\|_{q,\Omega}/|t-s|<\infty. \]

The filter $G$ models the non-stationarity of $(\eps_{i,n})$. The quantity $\delta_q(G,i)$ measures the dependence of $(\eps_{i,n})$ and plays a similar role as mixing coefficients. We now state some assumptions regarding the error terms in model \eqref{eq:model}.

\begin{assumption} \label{assumption:LS}
	Let the triangular array $\{(\eps_{i,n})_{1\leq i\leq n}\}_{n\in\N}$ in \eqref{eq:model} be centered and locally stationary
	 with filter $G$, such that the following conditions are satisfied:
	\begin{enumerate}
	\item There exists a constant $\gamma\in(0,1)$ such that $\delta_4(G,i)=\Oc(\gamma^i)$, as $i\to\infty$.
	\item The filter $G$ is Lipschitz continuous with respect to $\|\cdot\|_{4,\Omega}$ and $$\sup_{t\in[0,1]}\|G(t,\Fc_0)\|_{4,\Omega}<\infty.$$
	\item The  (local) \textit{long-run variance} of $G$, defined as
	\begin{equation} \label{hol6}
	\sigma^2(t)=\sum_{i=-\infty}^{\infty}\cov\big(G(t,\Fc_i),G(t,\Fc_0)\big) ,
	\end{equation}
	for $t\in[0,1]$, is Lipschitz continuous and bounded away from zero, i.\,e., $$\inf_{t\in[0,1]}\sigma^2(t)>0.$$
	\item The   moments of order $8$ are uniformly bounded, i.\,e., $\max_{1\leq i\leq n} \ex \eps_{i,n}^8 <\infty$.
	\end{enumerate}
\end{assumption}

\subsection{Testing for  relevant differences - the problem of estimating the variance} \label{sec23}

 Continuing the discussion in Section \ref{sec21}  it follows from the results  given in Section \ref{sec3} that
    the estimator \eqref{hol5}  is asymptotically normal distributed if
   Assumptions \ref{assumption:C2}, \ref{assump:kern}, \ref{assumption:LS}
 and an additional assumption  on the consistency of the statistic $\check g_{n}$ are satisfied. More precisely, it can be shown
 (see Remark \ref{hol10}) that
\begin{equation}\label{hol9}
\sqrt{n}\big(\|\check{\mu}_{h_n}-\check{g}_n\|_{2,\tau}^2 - \|\mu-g(\mu)\|_{2,\tau}^2 \big)\convw \Nc\big(0,4\|d_\omega\sigma\|_2^2\big),
\end{equation}
 where 
 the  symbol  $\convw$ denotes weak convergence,
 $\sigma^2 (\cdot) $ is the local long-run variance defined  in \eqref{hol6} and $d_\omega(\cdot)$ denotes  an unknown function, that depends on the function $\mu$ and the error process.
In principle, if $\hat{\sigma}_n^2$ and $\hat{d}_\omega^2$ are  estimators of the local long-run variance and the  function $d_\omega$, respectively, a reasonable strategy would be
  to reject the null hypothesis in \eqref{h0:rel}
  if
\begin{equation}\label{eq:testLRV}
\|\check{\mu}_{h_n}-\check{g}_n\|_{2,\tau}^2 > \Delta^2 + z_{1-\alpha}\frac{2 \|\hat{d}_\omega\hat{\sigma}_n\|_{2}}{\sqrt{n}},
\end{equation}
where $z_{1-\alpha}$ denotes the $(1-\alpha)$-quantile of the standard normal distribution. It will be shown in Remark \ref{hol10} below, that this decision rule provides a consistent and asymptotic level $\alpha$-test for the hypotheses in \eqref{h0:rel}. 
However, it turns out that  this decision rule does not provide a stable test because  local estimators of the long-run variance
have a  rather large variability.

In order to avoid the intricate estimation of the local long-run variance
we will re-define the local linear estimator in \eqref{hol8}  permuting the data and consider the partial sum process of the new estimators
in the following section.
This approach will enable us to  construct  an (asymptotically) pivotal test statistic for the hypotheses in \eqref{h0:rel}.

	\section{A pivotal test statistic} \label{sec3}

\subsection{Self-normalization} \label{sec31}

A common technique
to avoid estimating the long-run variance are ratio statistics or self-normalization  as first introduced by
\cite{horvath2008}  and  \cite{Shao2010b}, which are  based on a convenient rescaling of the test statistic.
However, these concepts are not easy to transfer to  non-stationary time series as they rely on the asymptotic
properties of  a  corresponding partial sum process.
To illustrate the problems of these concepts   in non-stationary time series consider  the simplest case of model \eqref{eq:model},
where the mean function is constant  and the error process is stationary.  In  this case the estimate  of the constant mean function
$\mu$ from the partial  sample
$X_{1,n},\ldots , X_{\lfloor \lambda n\rfloor  ,n}$ is its mean and  under the assumptions stated
in Section \ref{sec2} we have  the weak convergence
\begin{equation*}
\{
B_n(\lambda) \}_{\lambda \in [0,1]} = \Big \{  n^{-1/2}\sum_{i=1}^{\lfloor \lambda n \rfloor} \big(X_{i,n} - \mu \big) \Big  \}_{\lambda \in [0,1]} \convw \{  \sigma W(\lambda) \}_{\lambda \in [0,1]},
\end{equation*}
where $\big \{W(\lambda ) \big  \}_{\lambda\in[0,1]}$ denotes  a standard Brownian motion and the long-run variance $\sigma^2$
is defined in \eqref{hol6} and does not depend on $t$ (because of the  stationarity assumption).
In this case, the factorisation of the limit into the long-run variance and a probabilistic term is used to construct a test statistic in the form of a ratio,
 such that $\sigma$ occurs in the nominator and denominator, and therefore cancels out.
 On the other hand, if  the  error process in  model \eqref{eq:model}  is non-stationary (but the mean function is  still constant)
  we have the weak convergence
\begin{equation*}
\{ B_n(\lambda) \}_{\lambda \in [0,1]}     \convw \Big \{ \int_0^\lambda \sigma(u)\diff W(u) \big  \}_{\lambda \in [0,1]} .
\end{equation*}
In this case, the limiting distribution does not factorise and it is no longer possible to use the common self-normalization approach. \cite{zhao2013} and \cite{rho2015} discuss locally stationary time series, but the proposed self-normalization procedures have to be combined with a wild bootstrap.

In this work, we present an alternative  self-normalization procedure for non-stationary time series which does not require  resampling
to obtain (asymptotically) pivotal statistics.
Our approach is based  on the idea that in a locally stationary setting, observations  from
the whole interval $[0,1]$  need to be taken into account. Therefore, let $b_n$ denote a sequence with $b_n\to\infty$ and $b_n/n\to 0$, as $n\to\infty$, and let $\ell_n=\lfloor n/b_n\rfloor$. We define a (fixed)  permutation of the set  $\{1,\dots,n\}$ by
\[T: \left\{\begin{array}{ccl}\{1,\dots, n\}&\to &\{1,\dots,n\}\\ k & \mapsto & T_k=
\left\{\begin{array}{rl}(k-1 \mod \ell_n)b_n+\lceil k/\ell_n\rceil,&\text{if}~k\leq \ell_nb_n\\ k, & \text{if}~ k> \ell_nb_n \end{array}\right.  \end{array}\right.\]
Note that for $k=i\ell_n+j$ it holds $T_k=(j-1)b_n+i+1$, where $i\in \{0,\dots, b_n\}$ and $j\in\{1,\dots,\ell_n\}$.

Roughly speaking, the mapping $T$ splits the set $\{1,\dots,n\}$ into $\ell_n$ blocks with block length $b_n$, that is 
 \begin{eqnarray*}
   \{ T_1, \ldots, T_{\ell_n} \} &=& \{ 1, b_n+1, 2b_n+1, \ldots, (l_n-1)b_n+1 \} \\
   \{ T_{\ell_n+1}, \ldots, T_{2\ell_n} \} &=& \{ 2, b_n+2, 2b_n+2, \ldots, (l_n-1)b_n+2 \} \\
   \{ T_{2\ell_n+1}, \ldots, T_{3\ell_n} \} &=& \{ 3, b_n+3, 2b_n+2, \ldots, (l_n-1)b_n+3 \} \\
    \quad      \vdots   \quad     \quad     \quad       &\vdots  &  \quad  \quad    \quad    \quad      \quad    \quad     \quad \quad     \quad \vdots  
 \end{eqnarray*}
 where the blocks correspond to the columns in the above display.

With this notation, for $\zeta>0$ and $\lambda\in [\zeta,1]$, we define the sequential local linear estimator
of the mean function $\mu$ from the sample
$X_{T_1,n}, \ldots  , X_{T_{\lfloor \lambda n\rfloor},n}$ as the first coordinate of the vector
\begin{equation}\label{eq:defLocLinEst}
\big(\hat{\mu}_{h_n}(\lambda,t),\widehat{\mu'}_{h_n}(\lambda,t)\big) = \underset{b_0,b_1}{\argmin} \sum_{i=1}^{\lfloor \lambda n\rfloor} \big(X_{T_i,n}-b_0-b_1(T_i/n-t)\big)^2 K_{h_n}(T_i/n-t).
\end{equation}
In the following we will work with a bias corrected version of  $\hat{\mu}_{h_n}(\lambda,t)$ and consider
the  sequential Jackknife estimator
\begin{equation}\label{eq:defJackknife}
\tilde{\mu}_{h_n}(\lambda,t)=2 \hat{\mu}_{h_n/\sqrt{2}}(\lambda,t)-\hat{\mu}_{h_n}(\lambda,t).
\end{equation}
With the notation 
$$
d(x):=\mu(x)-g(\mu)
$$
we can rewrite  the distance in \eqref{hol4} as $d_0 = \|d\|_{2,\tau}$. In order to estimate $d_0$
 let $\hat{g}_n(\lambda)$ be a suitable sequential estimator of the benchmark $g(\mu)$ from the sample
$X_{T_1,n}, \ldots  , X_{T_{\lfloor \lambda n\rfloor},n}$ and define
$$
\hat{d}_n(\lambda,x)=\tilde{\mu}_{h_n}(\lambda,x)-\hat{g}_n(\lambda)
$$
 and
\begin{equation*}
\hat{d}_{2,n}(\lambda)= \|\hat{d}_n(\lambda,\cdot)\|_{2,\tau}.
\end{equation*}
Note that all estimates are calculated from a part of the permuted sample and that the statistic
$\hat{d}_{2,n}(1)$  estimates $d_{0}$ from the full sample $X_{1,n}, \ldots  , X_{n,n}$ and therefore coincides
with the estimator defined in \eqref{hol5}.
For the proofs of our main results we need an assumption regarding the precision of the  estimator
$\hat g_{n} (\cdot )$ of the benchmark, the bandwidth $h_{n}$  and the block length  $b_{n}$, which are given next.

\begin{assumption}\label{assump:gn}
	The sequential estimator $\hat{g}_n(\lambda)$ of the benchmark $g(\mu)$ admits a stochastic expansion
		$$ \lambda \sqrt{n}\big(\hat{g}_n(\lambda)-g(\mu)\big)=\frac{1}{\sqrt{n}}\sum_{i=1}^{\lfloor \lambda n\rfloor} \eps_{T_i,n} \omega_n(T_i/n)+o_\pr(1), $$
		uniformly with respect to $\lambda \in [\zeta, 1]$
		for some constant $\zeta \in (0,1)$ and functions $\omega_n,\omega\in \Lc^4([0,1])$ such that $\omega_n$ is Riemann-integrable for any $n\in\N$, $\|\omega_n-\omega\|_4\to 0$  and
		\begin{equation}\label{eq:omega}
		 \sum_{j=1}^{\ell_n}  \sum_{r = 1}^{b_n} \big|\omega_n\big(\tfrac{jb_n}{n}\big)-\omega_n\big(\tfrac{r+jb_n}{n}\big)  \big|=\Oc(b_nh_n^{-1}),
		\end{equation}
		where $\|\omega_n-\omega\|_4 = \big ( \int_0^1 \|\omega_n (x) -\omega (x) \|^4  d x \big )^{1/4}$.
\end{assumption}

\begin{assumption}\label{assump:seq}
    There exist  constants $\alpha, \beta>0$  such that the sequence of bandwidths
  $h_n\to 0$ satisfies   $nh_n\to\infty$, $nh_n^6\to0$, $n^\beta = \Oc(nh_n^4)$ and  the sequence $b_n\to\infty$    satisfies
 $b_n^3/n\to 0$, $\tfrac{b_n^2}{nh_n}\to 0$,  $n^\alpha=\Oc(b_n)$.
    \end{assumption}

\begin{remark}\label{rem:gn} ~

(1) 
Assumption \ref{assump:gn} is rather mild and satisfied for many functionals as explained below. Proofs of the following statements can be found in
Section \ref{sec:rem} of the Appendix.

	\begin{enumerate}
		\item[(i)] Condition
		\eqref{eq:omega} is satisfied for  all Lipschitz continuous functions and all  step functions
		on the interval $[0,1]$.
		\item[(ii)] The assumption holds for $g(\mu)=c$ with some known $c\in\R$ 
				 , for $\hat{g}_n(\lambda) = g(\tilde{\mu}_{h_n}(\lambda,\cdot))$.
		\item[(iii)] Assumption \ref{assump:gn} is satisfied for the functional defined in \eqref{hol3}
and the estimator
		$$
		\hat{g}_n(\lambda) =  \frac{1}{(t_1-t_0) \lambda n}\sum_{i=1}^{\lfloor \lambda n\rfloor} X_{T_i,n} \id(t_0\leq T_i/n\leq t_1) .
		$$
	\item[(iv)]
	Let $g:L^2([0,1])\to\R$ be a linear, bounded operator. By the Riesz-Fréchet representation theorem, there exists $\bar{h}_g\in L^2([0,1])$ such that $g(\cdot)=\langle \cdot, \bar{h}_g\rangle$. If there exists a continuous function $h_g$ in the equivalence class  corresponding to $\bar{h}_g$, the estimator $\hat{g}_n(\lambda) = g(\tilde{\mu}_{h_n}(\lambda,\cdot))$ satisfies Assumption \ref{assump:gn}.
	\item[(v)] The functional $g (\mu ) = \mu (t) $ (for some fixed $t\in [0,1]$) is not covered by Assumption \ref{assump:gn}.  Nevertheless a corresponding pivotal test can be developed as well - see Remark \ref{rem1} for more details.
\end{enumerate}
(2)
        Assumption \ref{assump:seq} is satisfied,   if $h_n=n^{-1/5}, b_n=n^{1/4}$. In this case, the constants $\alpha$ and $\beta$ can be chosen as $\frac{1}{4}$ and $\frac{1}{5}$, respectively.
    \end{remark}

\medskip

\begin{theorem}\label{thm:convSN}
	Let Assumptions \ref{assumption:C2}, \ref{assump:kern}, \ref{assumption:LS}, \ref{assump:gn} and \ref{assump:seq} be satisfied.
	For any $\zeta \in (0,1)$,
	the process
	\begin{equation}\label{eq:Gn}
	\{ G_n(\lambda) \}_{\lambda \in [\zeta,1] }
	 =\{ \lambda \sqrt{n}( \hat{d}_{2,n}^2(\lambda)-d_0^2 ) \}_{\lambda \in [\zeta,1] }
	\end{equation}
	converges weakly to  the process
		\begin{equation}\label{hol51}
		 \{ G(\lambda)\}_{\lambda \in [\zeta,1] } =\{ 2 \big\|d_\omega\sigma\big\|_{2} W(\lambda)
	\}_{\lambda \in [\zeta,1] }
	\end{equation}
	 in $\ell^\infty([\zeta,1])$, where $d_\omega(\cdot) = f_\tau(\cdot)d(\cdot) + \omega(\cdot) \int_0^1 d(x)\tau(\diff x)$ and
	$\{ W(\lambda)\}_{\lambda\in[0,1]}$ denotes a standard Brownian motion. In particular, $G(\lambda)=0$ if $d_0=0$.
\end{theorem}

\medskip

\begin{remark} \label{rem1}
	If $g(\mu)=\mu(t)$, for some fixed $t\in[0,1]$, the benchmark $g(\mu)$ needs to be estimated locally and there is no estimator 
satisfying  Assumption \ref{assump:gn}. 
 However, an analogous result as stated in Theorem \ref{thm:convSN}  can be 
  shown with the same arguments given in the proof of  the latter theorem.
  More precisely, if $g(\mu)=\mu(t)$ we can use  $\hat{g}_n(\lambda) = \tilde{\mu}_{h_n}(\lambda,t)$ and under Assumptions \ref{assumption:C2}, \ref{assump:kern}, \ref{assumption:LS} and \ref{assump:seq}, the process 
	$$\{ G_n'(\lambda) \}_{\lambda \in [\zeta,1] }
	=\{ \lambda \sqrt{nh_n}( \hat{d}_{2,n}^2(\lambda)-d_0^2 ) \}_{\lambda \in [\zeta,1] }$$
	converges weakly to
	$$ \{ G'(\lambda)\}_{\lambda \in [\zeta,1] } =\bigg\{ 2 \sigma(t)  \kappa(t) \int_0^1 d(x)\tau(\diff x)  W(\lambda)\bigg\}_{\lambda \in [\zeta,1] }$$  in $\ell^\infty([\zeta,1])$, where
	the constant $\kappa$ is defined  by   $\kappa^2(t)=\int_{-1}^1 \big(K^*(x)\big)^2 \diff x$ if  $t\in (0,1)$ and by 
	$$ \kappa^2(t) = \frac{1}{(\kappa_{t,0}\kappa_{t,2}-\kappa_{t,1}^2)^2} \int_{-t}^{1-t} \bigg\{\bigg( \frac{\kappa_{t,2}}{\sqrt{2}}-\kappa_{t,1} x \bigg) K^*(x) + \bigg(\frac{1}{\sqrt{2}}-1\bigg)\kappa_{t,2} K(x)\bigg\}^2 \diff x, $$
	with $\kappa_{t,j} = \int_{-t}^{1-t}x^j K(x)\diff x$, for $j\in\{0,1,2\}$ and $t\in \{0,1\}$, and $K^*$ is defined by  $ K^*(x)=2\sqrt{2}K(\sqrt{2}x)-K(x)$. 
\end{remark}

\medskip

In the following, we will develop  a pivotal test for the hypotheses \eqref{h0:rel} on the basis of Theorem, \ref{thm:convSN} or Remark  \ref{rem1}. For this purpose
let $\nu$ be a probability measure on the interval $[\zeta, 1]$ with $\nu(\{1\})=0$. We propose to reject the null hypothesis if
\begin{equation}\label{eq:test}
\hat{d}_{2,n}^2(1) > \Delta^2 + q_{1-\alpha}\int_\zeta^1 \lambda |\hat{d}_{2,n}^2(\lambda) -\hat{d}_{2,n}^2(1)|\diff \nu(\lambda), \end{equation}
where $q_{1-\alpha}$ denotes the $(1-\alpha)$-quantile of the distribution of the random variable
$$
\frac{W(1)}{\int_\zeta^1 |W(\lambda)-\lambda W(1) |\diff \nu(\lambda)}.
$$

\medskip

\begin{corollary} \label{cortest}
Let the assumptions of either Theorem \ref{thm:convSN} or of Remark \ref{rem1} be satisfied.
	If  $\Delta>0$, the decision rule \eqref{eq:test} defines a consistent and asymptotic level $\alpha$-test for the hypotheses \eqref{h0:rel} of a relevant deviation of the mean function $\mu$ from the threshold $g(\mu)$, that is
	$$ \pr\big(~
\mbox{\rm the null hypothesis \eqref{h0:rel} is rejected by  \eqref{eq:test} }
	\big) \xrightarrow{n\to\infty} \left\{\begin{array}{rl}0,&\text{if}~d_0<\Delta\\ \alpha,&\text{if}~d_0=\Delta\\ 1,&\text{if}~d_0>\Delta. \end{array}\right.
$$
\end{corollary}

\medskip

\begin{remark} \label{hol10}
Note that the Jackknife estimator defined in  \eqref{hol8a} coincides with $\tilde{\mu}_{h_n}(1,\cdot)$ and $\check  g_n = g( \tilde{\mu}_{h_n}(1,\cdot))$.
Consequently, the continuous mapping theorem and  Theorem \ref{thm:convSN}  yield the weak convergence stated in equation \eqref{hol9} of Section \ref{sec23}.
Consequently, if the estimators $\hat{\sigma}_n^2$ and $\hat{d}_\omega^2$ are consistent, the decision  rule in \eqref{eq:testLRV} defines a consistent and asymptotic level $\alpha$ test for the hypothesis \eqref{h0:rel}, that is
 	$$ \pr\big(~
\mbox{\rm the null hypothesis \eqref{h0:rel} is rejected by  \eqref{eq:testLRV} }
	\big)
	 \xrightarrow{n\to\infty} \left\{\begin{array}{rl}0,&\text{if}~d_0<\Delta\\ \alpha,&\text{if}~d_0=\Delta\\ 1,&\text{if}~d_0>\Delta. \end{array}\right.
$$
\end{remark}

\section{Finite Sample Properties} \label{sec4}

\subsection{Monte Carlo simulation study}

A large scale Monte Carlo simulation study was performed to analyse the finite-sample properties of the proposed test \eqref{eq:test}.
 The local linear estimator in \eqref{eq:defLocLinEst} requires the specification of the kernel $K$ and the bandwidth $h_n$.
 We used the quartic kernel $K(x)=\tfrac{15}{16}(1-x^2)^2$, but other kernels will yield similar results.
The choice of the bandwidth $h_n$ for the estimator $\tilde{\mu}_{h_n}$ is crucial to avoid both overfitting and oversmoothing,
and we employ the following $k$-fold cross-validation procedure with $k=10$ (as recommended by \citealp{HasTibFri09}, page 242).
\begin{alg}[Cross-Validation for the Choice of $h_n$] \label{alg:hn}~
	\begin{compactenum}
		\item Split the observed data randomly in $k=10$ sets $S_1, \dots, S_{10}$ of equal length.	
		\item For $h_n=\tfrac{1}{n}$ and each set $S_i$, calculate the Jackknife estimator $\tilde{\mu}_{h_n}^{(i)}$ based on the data in the remaining sets.	
		\item Based on the Jackknife estimators $\tilde{\mu}_{h_n}^{(i)}$ from Step (2), compute the mean squared prediction error
		\[
		\MSE_{h_n}=\frac{1}{1-h_n} \sum_{i=1}^{10}\sum_{j\in S_i} \big\{ X_{j,n}-\tilde{\mu}_{h_n}^{(i)}(j/n)\big\}^2.
		\]		
		\item Repeat Steps (2) and (3) for the bandwidths $h_n=\tfrac{2}{n},\dots,\tfrac{\lfloor n/2 \rfloor}{n}$		
		\item Choose the bandwidth $h_n$ that minimises the mean squared prediction error $\MSE_{h_n}$.	
	\end{compactenum}
\end{alg}
 As block width we chose $b_n=20$ and as measure $\nu$ on $[0,1]$ in \eqref{eq:test} we used the uniform
 distribution on the  set $\{1/5, \ldots ,4/5\}$.
  Preliminary simulation studies showed that different choices of    $b_n$ and the measure $\nu$ lead to similar results.

We considered two types of mean functions $\mu$, three different error processes  and four different choices of the time-dependent variance.
The first class of models is based on the mean function
\begin{equation}
\label{m1}
\mu_a^{(1)}(x) = 10 + \tfrac{1}{2} \sin(8\pi x)+a\big(x-\tfrac{1}{4}\big)^2 \id\big(x>\tfrac{1}{4}\big),
\end{equation}
which is displayed in the left part of  Figure \ref{fig:models} for various choices of the parameter $a$. We considered the testing problem
\begin{equation}
\label{h0avlspec}
H_0:    d_0 := \Big \| \mu_{a}^{(1)} - g(\mu)  \Big \|_{2,\tau}  \leq 1/2   \quad
\text{  vs.\ }  \quad H_1:  \Big \| \mu_{a}^{(1)} - \bar{\mu}_a^{(1)} \Big \|_{2,\tau}  > 1/2,
\end{equation}
where 
$$
g(\mu) = 
\bar{\mu}_a^{(1)}=2\int_0^{1/2} \mu_{a}^{(1)}(x)\diff x, 
$$ 
and $\tau(\cdot)=2\lambda_{[1/2,1]}(\cdot) $ is  the Lebesgue measure on the interval $[\frac{1}{2},1]$. Such a scenario might for instance be encountered and of interest in the context of analyzing climate data where measurements for a recent period are compared with an average from previous years.

Note that $\big \| \mu_{a}^{(1)} - \bar{\mu}_a^{(1)} \big \|_{2,\tau }=1/2$ for $a^*\approx 1.43$.
We call this situation (i.e. when there is equality in \eqref{h0avlspec}) the {\it boundary of the hypotheses}.
On the other hand for $a<a^*$ and $a>a^*$
the null hypothesis and  alternative
 in \eqref{h0avlspec} are satisfied, respectively.

\begin{figure}[tbp]
	\centering \vspace{-.4cm}
	\mbox{\hspace{-.4cm}\includegraphics[width=65mm]{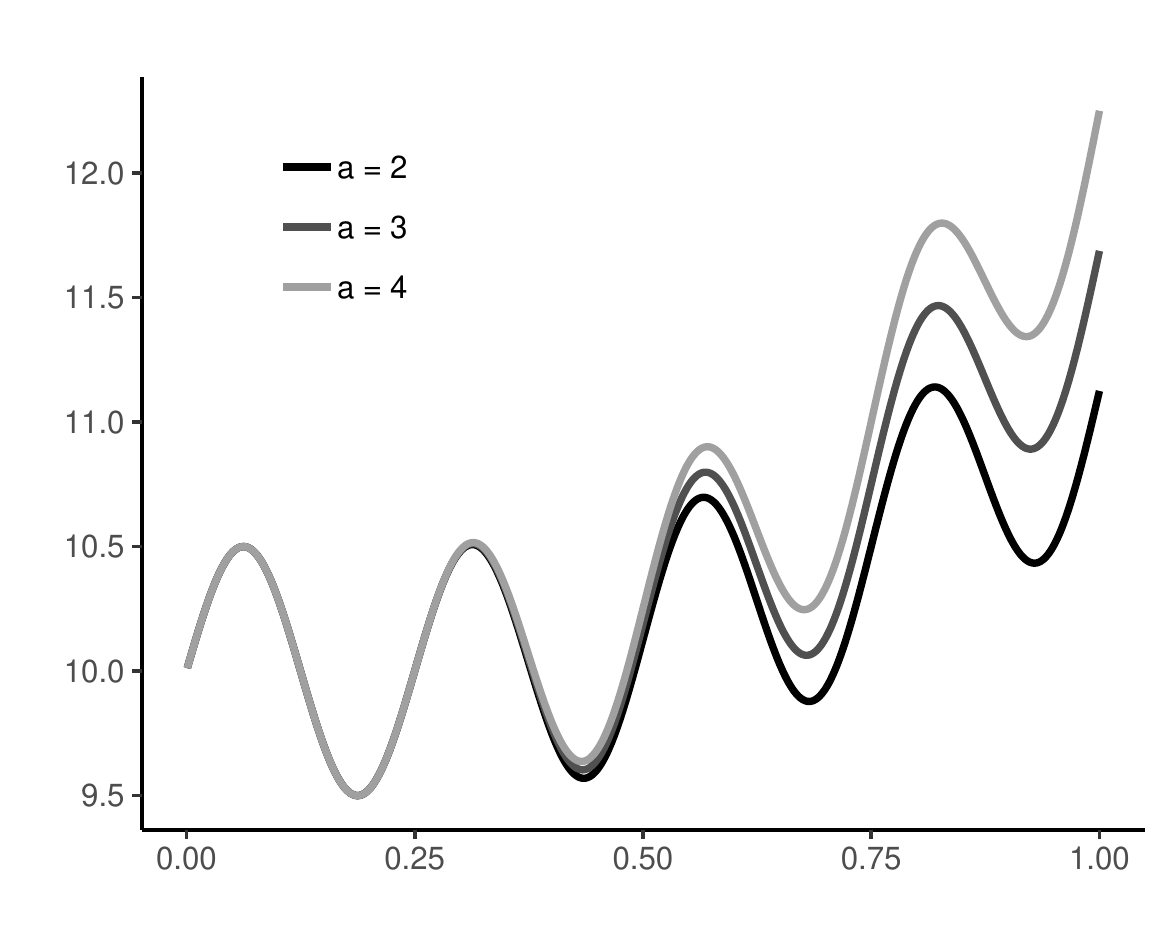} \hspace{-.5cm}
		\includegraphics[width=65mm]{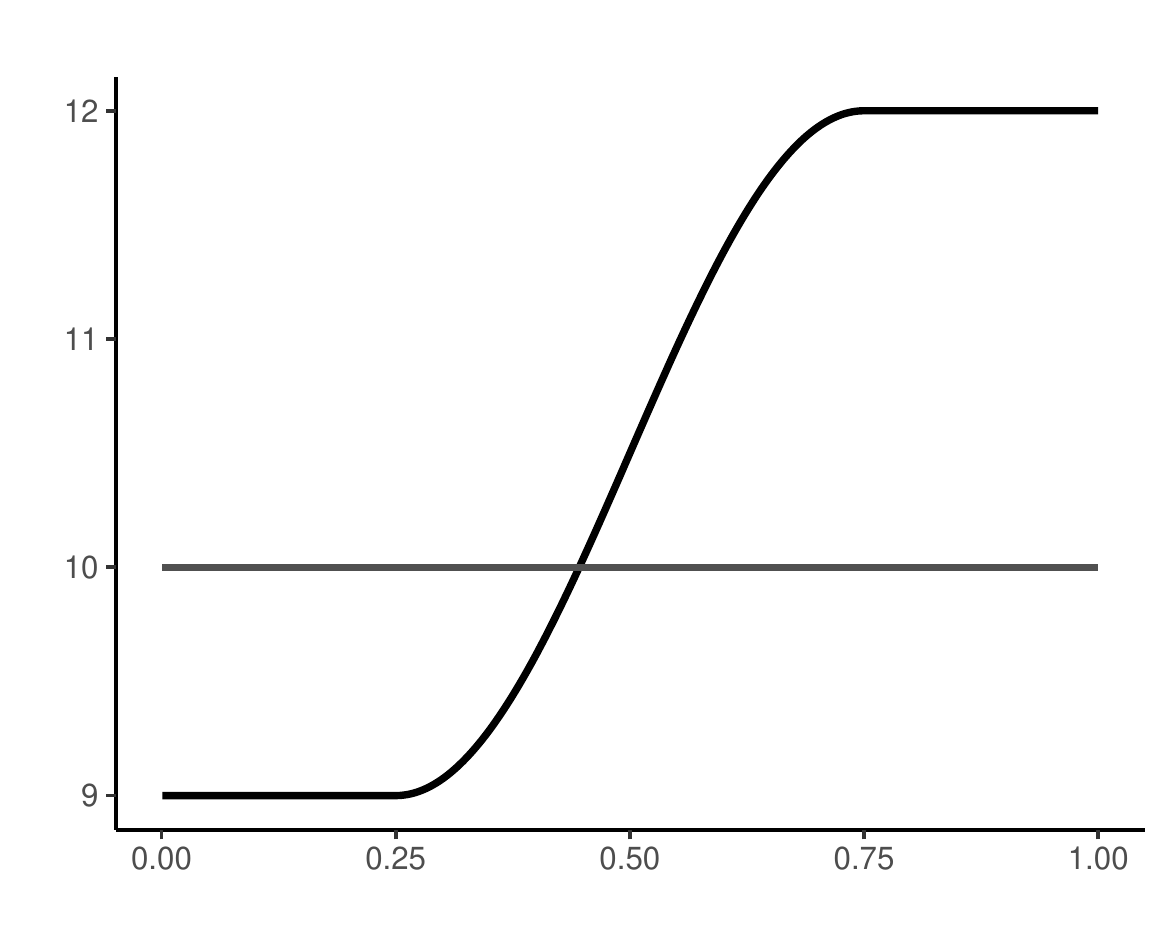}}\vspace{-.5cm}
	\caption{\textit{Left: The mean function $\mu_a^{\scs (1)}$  for three choices of $a$. Right: The mean function $\mu^{\scs (2)}$.}}
	\label{fig:models}
\end{figure}

The second  model has  the mean function
\begin{equation}
\label{m2}
 \mu^{(2)}(x) = \left\{\begin{array}{ll}
9&\text{for}~x\le \tfrac{1}{4}\\
- \tfrac{3}{2}\sin(2\pi x)+10.5&\text{for}~\tfrac{1}{4}<x\le \tfrac{3}{4}\\
12&\text{for}~\tfrac{3}{4}<x
\end{array} \right.
\end{equation}
which is displayed  in the right  part of   Figure \ref{fig:models}. For models involving this mean function, we considered the testing problem 
\begin{equation}
\label{h0allspec}
H_0:   d_0 =   \| \mu^{(2)} - g(\mu)\|_{2,\tau}  \leq \Delta
\quad
\text{   vs.\  }   \quad H_1:  \| \mu^{(2)} - g(\mu) \|_{2,\tau}   >   \Delta~
\end{equation}
for
 various choices of the threshold $\Delta > 0$,
where  $g(\mu)\equiv 10 $ and  $\tau(\cdot)=\lambda_{[0,1]}(\cdot)$ is the Lebesgue measure on
the interval  $[0,1]$.
Such a setting might be encountered in quality control, where deviations from a target value might occur gradually due to wear and tear (and eventual failure) of a component of a complex system. Note that  $\| \mu^{(2)} - 10 \|_{2,\tau} \le\Delta$ for $\Delta \ge 1.392$, whereas $\| \mu^{(2)} - 10 \|_{2,\tau}>\Delta$ for $\Delta < 1.392$.

{\footnotesize

	\begin{table}[t!]
		\begin{tabular}{l|r | rrr | rrr | rrr | rrr}		\hline \hline
			$\mu_a^{(1)}$&&\multicolumn{3}{c|}{$\tilde\sigma^2_0$}&\multicolumn{3}{c|}{$\tilde\sigma^2_1$} &\multicolumn{3}{c|}{$\tilde\sigma^2_2$}&\multicolumn{3}{c}{$\tilde\sigma^2_3$}\\
			$a$ &$d_0-\tfrac{1}{2}$ & 200 & 500 & 1000 & 200 & 500 & 1000 & 200 & 500 & 1000 & 200 & 500 & 1000 \\ 
			\hline 
			\addlinespace[.2cm]
			\multicolumn{11}{l}{\quad\textit{Panel A: iid errors}} \\ 
			0.37 & -0.15 & 0.0 & 0.0 & 0.0 & 0.0 & 0.0 & 0.0 & 0.0 & 0.0 & 0.0 & 0.0 & 0.2 & 0.0 \\ 
			0.89 & -0.10 & 0.0 & 0.1 & 0.0 & 0.0 & 0.2 & 0.0 & 0.0 & 0.0 & 0.0 & 0.0 & 0.2 & 0.1 \\ 
			1.18 & -0.05 & 0.2 & 0.9 & 0.6 & 0.0 & 1.0 & 0.3 & 0.0 & 1.4 & 0.4 & 0.1 & 0.9 & 0.6 \\ 
			\textbf{1.43} & \textbf{0.00} & \textbf{0.7} & \textbf{3.9} & \textbf{4.5} & \textbf{0.4} & \textbf{2.8} & \textbf{3.4} & \textbf{0.5} & \textbf{5.5} & \textbf{4.6} & \textbf{0.3} & \textbf{2.7} & \textbf{3.4} \\ 
			1.86 & 0.10 & 2.7 & 23.6 & 43.1 & 2.1 & 21.6 & 33.5 & 3.0 & 32.2 & 47.0 & 1.0 & 13.4 & 26.5 \\ 
			2.26 & 0.20 & 7.4 & 54.6 & 83.9 & 8.2 & 46.3 & 73.0 & 9.5 & 66.5 & 88.6 & 3.5 & 33.4 & 61.4 \\ 
			2.64 & 0.30 & 14.1 & 76.6 & 95.8 & 12.1 & 68.7 & 91.8 & 19.5 & 88.8 & 98.0 & 6.2 & 51.1 & 83.1 \\ 
			 
			\multicolumn{11}{l}{\quad\textit{Panel B: MA errors}} \\ 
			0.37 & -0.15 & 0.0 & 0.0 & 0.0 & 0.0 & 0.0 & 0.0 & 0.0 & 0.0 & 0.0 & 0.0 & 0.0 & 0.0 \\ 
			0.89 & -0.10 & 0.0 & 0.0 & 0.0 & 0.0 & 0.0 & 0.0 & 0.0 & 0.0 & 0.0 & 0.0 & 0.0 & 0.0 \\ 
			1.18 & -0.05 & 0.1 & 0.4 & 0.2 & 0.0 & 0.8 & 0.1 & 0.0 & 0.2 & 0.2 & 0.1 & 0.2 & 0.4 \\ 
			\textbf{1.43} & \textbf{0.00} & \textbf{0.3} & \textbf{4.8} & \textbf{4.9} & \textbf{0.1} & \textbf{4.1} & \textbf{4.2} & \textbf{0.5} & \textbf{7.0} & \textbf{5.3} & \textbf{0.1} & \textbf{4.4} & \textbf{4.1} \\ 
			1.86 & 0.10 & 4.3 & 38.3 & 58.7 & 3.7 & 30.9 & 53.3 & 5.5 & 45.4 & 69.0 & 3.0 & 21.9 & 39.2 \\ 
			2.26 & 0.20 & 12.6 & 76.8 & 94.8 & 11.6 & 69.3 & 91.4 & 16.3 & 83.8 & 97.8 & 9.4 & 52.5 & 79.5 \\ 
			2.64 & 0.30 & 29.5 & 93.7 & 99.9 & 28.6 & 87.6 & 98.8 & 33.4 & 97.4 & 99.7 & 21.4 & 75.4 & 94.0 \\ 

			\multicolumn{11}{l}{\quad\textit{Panel C: AR errors}} \\  
			0.37 & -0.15 & 0.0 & 0.0 & 0.0 & 0.0 & 0.0 & 0.0 & 0.0 & 0.0 & 0.0 & 0.1 & 0.0 & 0.0 \\ 
			0.89 & -0.10 & 0.0 & 0.3 & 0.0 & 0.2 & 0.1 & 0.0 & 0.1 & 0.3 & 0.0 & 0.3 & 0.6 & 0.1 \\ 
			1.18 & -0.05 & 0.4 & 1.3 & 1.1 & 0.1 & 1.5 & 0.9 & 0.0 & 2.3 & 1.3 & 0.4 & 2.4 & 1.3 \\ 
			\textbf{1.43} & \textbf{0.00} & \textbf{1.3} & \textbf{8.6} & \textbf{6.5} & \textbf{1.2} & \textbf{7.1} & \textbf{7.5} & \textbf{1.0} & \textbf{7.8} & \textbf{7.0} & \textbf{0.8} & \textbf{5.6} & \textbf{5.8} \\ 
			1.86 & 0.10 & 5.3 & 36.3 & 52.8 & 5.4 & 29.5 & 47.9 & 6.8 & 41.1 & 57.0 & 3.9 & 22.6 & 35.4 \\ 
			2.26 & 0.20 & 16.4 & 67.8 & 90.0 & 14.6 & 59.9 & 85.7 & 18.0 & 78.2 & 92.6 & 10.0 & 45.0 & 70.1 \\ 
			2.64 & 0.30 & 30.6 & 86.6 & 98.9 & 24.1 & 80.1 & 96.1 & 35.4 & 92.9 & 99.4 & 18.7 & 65.6 & 89.5 \\ 
			\hline \hline
		\end{tabular} \medskip
		\caption{\it Empirical rejection rates of the test \eqref{eq:test}  for the hypotheses  \eqref{h0avlspec}. The mean function is given by 
		\eqref{m1}, where   different values for the parameter $a$, different error processes, and sample sizes are considered.
		The lines in boldface correspond to  the boundary of the hypotheses. }
		\label{tab:mu1}
	\end{table}}

We consider  four different choices of time-dependent variance $\tilde{\sigma}^2(t)= \ex[G^2(t,\Fc_0)]$, that is
\begin{align*}
\tilde\sigma_0^2(t) &= 1, &  	
\tilde\sigma_1^2(t) &= \tfrac12 + t, \\
\tilde\sigma_2^2(t) &=1-\tfrac{1}{2}\cos(2\pi t),   &
\tilde\sigma_3^2(t) &= \tfrac{1}{2}+\id(t \geq 1/2)~,
\end{align*}
and  three classes of error processes $\{\eps_{i,n}:1\leq i\leq n\}_{n\in\N}$ in model \eqref{eq:model}, that is
\begin{align*}
(\text{IID}) \quad &~ \eps_{i,n}= \tilde{\sigma}_k(i/n)\eta_i\\
(\text{MA)} \quad &~ \eps_{i,n} = \tilde{\sigma}_k(i/n)\big( \eta_i + \tfrac{1}{2}\eta_{i-1}\big)/2\\
(\text{AR})  \quad &~ \eps_{i,n} = \tilde{\sigma}_k(i/n) \big( \eta_i + \tfrac{1}{2}\eps_{i-1,n}\big)/2,
\end{align*}
for $k\in\{0,1,2,3\}$, where $(\eta_i)_{i\in\Z}$ is an i.i.d.\ sequence of standard normal distributed random variables.




{\footnotesize
	\begin{table}[t!]
		\begin{tabular}{l|r | rrr | rrr | rrr | rrr}		\hline \hline
			$\mu^{(2)}$&&\multicolumn{3}{c|}{$\tilde\sigma^2_0$}&\multicolumn{3}{c|}{$\tilde\sigma^2_1$} &\multicolumn{3}{c|}{$\tilde\sigma^2_2$}&\multicolumn{3}{c}{$\tilde\sigma^2_3$}\\
			$\Delta$ &$\Delta-d_0$ & 200 & 500 & 1000 & 200 & 500 & 1000 & 200 & 500 & 1000 & 200 & 500 & 1000 \\ 
			\hline 
			\addlinespace[.2cm]
			\multicolumn{11}{l}{\quad\textit{Panel A: iid errors}} \\
			1.30 & -0.09 & 13.8 & 41.8 & 62.8 & 12.5 & 36.3 & 59.4 & 14.8 & 47.0 & 68.4 & 9.1 & 23.4 & 44.3 \\ 
			1.34 & -0.05 & 6.8 & 20.8 & 34.3 & 6.9 & 19.1 & 29.8 & 7.2 & 22.8 & 37.6 & 4.8 & 11.7 & 23.3 \\ 
			1.38 & -0.01 & 2.1 & 8.9 & 11.2 & 4.0 & 8.2 & 8.3 & 3.1 & 8.4 & 11.2 & 2.7 & 4.8 & 7.9 \\ 
			\bf 1.39 & \bf  0.00 & \bf  2.0 & \bf  5.0 & \bf  5.1 & \bf  2.2 & \bf  5.2 & \bf  5.4 & \bf  2.4 & \bf  5.1 & \bf  5.8 & \bf  2.4 & \bf  3.3 & \bf  4.5 \\ 
			1.41 & 0.02 & 1.0 & 2.6 & 1.1 & 1.8 & 3.0 & 1.0 & 1.2 & 2.6 & 1.5 & 1.8 & 2.0 & 2.4 \\ 
			1.48 & 0.09 & 0.3 & 0.0 & 0.0 & 0.1 & 0.0 & 0.0 & 0.1 & 0.1 & 0.0 & 0.5 & 0.5 & 0.0 \\ 
			1.55 & 0.16 & 0.1 & 0.0 & 0.0 & 0.0 & 0.0 & 0.0 & 0.0 & 0.0 & 0.0 & 0.1 & 0.0 & 0.0 \\ 
						\multicolumn{11}{l}{\quad\textit{Panel B: MA errors}} \\ 
			1.30 & -0.09 & 27.3 & 58.0 & 84.6 & 20.7 & 54.4 & 79.7 & 32.5 & 67.6 & 86.3 & 15.8 & 40.2 & 66.1 \\ 
			1.34 & -0.05 & 12.6 & 30.7 & 52.3 & 9.5 & 26.9 & 45.9 & 15.2 & 37.2 & 55.7 & 7.2 & 18.8 & 34.5 \\ 
			1.38 & -0.01 & 5.5 & 11.1 & 13.7 & 3.8 & 8.5 & 12.6 & 5.3 & 11.1 & 14.8 & 2.8 & 7.6 & 11.7 \\ 
			\bf 1.39 & \bf  0.00 & \bf  3.0 & \bf  5.0 & \bf  7.7 & \bf  3.4 & \bf  5.1 & \bf  6.2 & \bf  3.6 & \bf  6.9 & \bf  7.5 & \bf  2.2 & \bf  4.4 & \bf  5.9 \\
			1.41 & 0.02 & 1.6 & 2.0 & 1.5 & 1.3 & 1.2 & 0.9 & 1.0 & 2.0 & 0.8 & 1.3 & 2.3 & 1.5 \\ 
			1.48 & 0.09 & 0.1 & 0.0 & 0.0 & 0.0 & 0.0 & 0.0 & 0.0 & 0.0 & 0.0 & 0.0 & 0.0 & 0.0 \\ 
			1.55 & 0.16 & 0.0 & 0.0 & 0.0 & 0.0 & 0.0 & 0.0 & 0.0 & 0.0 & 0.0 & 0.0 & 0.0 & 0.0 \\ 
			\multicolumn{11}{l}{\quad\textit{Panel C: AR errors}} \\ 
			1.30 & -0.09 & 23.8 & 52.2 & 75.0 & 20.7 & 43.5 & 67.8 & 26.6 & 58.6 & 79.2 & 15.1 & 36.7 & 53.6 \\ 
			1.34 & -0.05 & 13.9 & 28.3 & 42.8 & 11.8 & 23.2 & 39.3 & 13.2 & 32.8 & 50.1 & 9.3 & 22.0 & 31.1 \\ 
			1.38 & -0.01 & 7.9 & 10.1 & 15.3 & 6.3 & 8.9 & 12.8 & 5.8 & 12.7 & 16.8 & 6.0 & 10.1 & 12.5 \\ 
			\bf 1.39 & \bf  0.00 & \bf  5.7 & \bf  7.5 & \bf  9.0 & \bf  4.5 & \bf  7.5 & \bf  8.6 & \bf  6.3 & \bf  8.0 & \bf  9.6 & \bf  4.2 & \bf  7.3 & \bf  9.9 \\
			1.41 & 0.02 & 3.2 & 2.6 & 2.5 & 2.9 & 2.4 & 2.1 & 2.2 & 3.3 & 2.1 & 3.5 & 3.7 & 3.0 \\ 
			1.48 & 0.09 & 0.6 & 0.0 & 0.0 & 0.2 & 0.0 & 0.0 & 0.3 & 0.1 & 0.0 & 0.8 & 0.3 & 0.1 \\ 
			1.55 & 0.16 & 0.1 & 0.0 & 0.0 & 0.0 & 0.0 & 0.0 & 0.0 & 0.0 & 0.0 & 0.1 & 0.0 & 0.0 \\ 
			\hline \hline
		\end{tabular} \medskip
		\caption{\it Empirical rejection rates of the test \eqref{eq:test}  for the hypotheses  \eqref{h0allspec}. The mean function is given by 
		\eqref{m2}, where   different values for the threshold  $\Delta$, different error processes, and sample sizes are considered.
		The lines in boldface correspond to  the boundary of the hypotheses. 
 }
		\label{tab:mu2}
	\end{table}

}

The empirical rejection rates of the test  \eqref{eq:test}  for the hypotheses
$H_0: d_0\leq\Delta$  vs. $H_0: d_0 > \Delta$ are calculated by $N=1000$ simulation runs
and displayed in Table~\ref{tab:mu1} and Table~\ref{tab:mu2}.   The sample size is chosen as $n = 200,$ $500$ and $1000$ and the nominal level is $5\%$.
Table~\ref{tab:mu1}  shows the rejection probabilities
for different values of $a$ in the function $ \mu_{a}^{(1)}$ defined in \eqref{m1}, which yields to different values of $d_0$
in the hypotheses \eqref{h0avlspec}. On the other hand, in
 Table~\ref{tab:mu2} the function  $\mu^{(2)}$ and therefore the value $d_0$ is fixed and the threshold $\Delta$  in the hypotheses is varied.
The lines marked in boldface indicate the boundary of the null hypothesis, that is, the parameter where $d_0 = \Delta$.
More precisely,  note that the null hypothesis in  \eqref{h0avlspec} holds
 if and only if $d_0 \leq 0.5$ and we display exemplary results for the cases $d_0=0.35$, $0.4$, $0.45$ and $0.5$ in Table~\ref{tab:mu1},
 where the last case corresponds to the boundary of the null hypotheses.  The remaining cases $d_0=0.6$, $0.7$ and $0.8$
 represent three scenarios of the alternative in  \eqref{h0avlspec}.
Similarly, in Table~\ref{tab:mu2} the function  $\mu^{(2)}$ is fixed with $d_0 =  1.39$. Therefore, the null hypothesis  in \eqref{h0allspec} holds
if and only if  the threshold satisfies  $\Delta \ge 1.39$.  We observe in most  cases a  good  approximation of the nominal level 
at the boundary of the hypotheses and the test is also able to detect alternatives
with reasonable power. These  empirical findings  corresponds with the theoretical results derived in Section \ref{sec3}.

We conclude this section with a comparison of the new test \eqref{eq:test} with the test \eqref{eq:testLRV}  which relies on the estimation
of the (local) long-run variance.    For this purpose
we use the long-run variance estimator as proposed in  equation
(4.7) of \cite{dettewu2019} with bandwidths as suggested in this reference. In Table \ref{tab:comp} and \ref{tab:comp2} we display the rejection
probabilities for both tests for some of the models  considered in Table~\ref{tab:mu1} and Table~\ref{tab:mu2}, where we use the Lebesgue measure on the interval
$[0,1]$ for the calculation of the $L^2$-distances and the benchmark is given by $g(\mu ) = \int_0^1 \mu (x) \diff x$.
For the sake of brevity we restrict
ourselves to the sample size $n=500$ and the variance function $\tilde{\sigma}_0^2(t)=1$.
We observe that  the test \eqref{eq:testLRV} is conservative at the boundary of the hypotheses. As a consequence
the proposed test \eqref{eq:test} based on self-normalization is usually more powerful.

{\footnotesize
\begin{table}[t!]
	\begin{tabular}{l|l | rr | rr | rr }		\hline \hline 
		& errors & \multicolumn{2}{c|}{i.i.d} & \multicolumn{2}{c|}{MA} & \multicolumn{2}{c}{AR}\\
		$a$ & $d_0-\Delta$ & \eqref{eq:test} & \eqref{eq:testLRV} & \eqref{eq:test} & \eqref{eq:testLRV} & \eqref{eq:test} & \eqref{eq:testLRV} \\ 
		\hline
		0.13 & -0.15 & 0.0 & 0.0 & 0.0 & 0.0 & 0.0 & 0.0 \\ 
		1.60 & -0.10 & 0.0 & 0.0 & 0.0 & 0.0 & 0.0 & 0.0 \\ 
		2.13 & -0.05 & 0.0 & 0.0 & 0.1 & 0.0 & 0.1 & 0.0 \\ 
		\textbf{2.57} & \textbf{0.00} & \textbf{2.2} & \textbf{0.0} & \textbf{1.6} & \textbf{0.0} & \textbf{4.3} & \textbf{0.4} \\ 
		2.97 & 0.05 & 14.0 & 3.0 & 20.1 & 0.9 & 22.6 & 3.5 \\ 
		3.35 & 0.10 & 35.0 & 24.2 & 58.1 & 20.5 & 46.3 & 26.8 \\ 
		3.71 & 0.15 & 61.9 & 67.1 & 84.8 & 75.7 & 72.9 & 71.7 \\ 
		\hline \hline
	\end{tabular} \medskip
	\caption{\it Empirical rejection rates of tests \eqref{eq:test} and \eqref{eq:testLRV} for the hypotheses  \eqref{h0avlspec}.
	The mean function is given by 
	\eqref{m1}, where   different values for the parameter $a$ and  different error processes  are considered. The variance is $\tilde{\sigma}_0^2(t)=1$,  the  sample size  is $n=500$ and the line in boldface corresponds to  the boundary of the hypotheses. }
	\label{tab:comp}
\end{table}
	
	}

{\footnotesize

\begin{table}[t!]
	\begin{tabular}{l|l | rr | rr | rr }		\hline \hline 
		& errors & \multicolumn{2}{c|}{i.i.d} & \multicolumn{2}{c|}{MA} & \multicolumn{2}{c}{AR}\\
		$\Delta$ & $\Delta-d_0$ & \eqref{eq:test} & \eqref{eq:testLRV} & \eqref{eq:test} & \eqref{eq:testLRV} & \eqref{eq:test} & \eqref{eq:testLRV} \\ 
		\hline
		1.15 & -0.15 & 73.8 & 83.9 & 90.3 & 93.1 & 81.8 & 86.2 \\ 
		1.20 & -0.10 & 48.0 & 45.0 & 69.3 & 51.7 & 58.3 & 50.4 \\ 
		1.25 & -0.05 & 22.9 & 10.5 & 32.1 & 6.7 & 31.3 & 12.4 \\ 
		\textbf{1.30} & \textbf{0.00} & \textbf{5.5} & \textbf{0.5} & \textbf{4.9} & \textbf{0.1} & \textbf{8.0} & \textbf{0.7} \\ 
		1.35 & 0.05 & 0.5 & 0.0 & 0.1 & 0.0 & 1.4 & 0.1 \\ 
		1.40 & 0.10 & 0.0 & 0.0 & 0.0 & 0.0 & 0.1 & 0.0 \\ 
		1.45 & 0.15 & 0.0 & 0.0 & 0.0 & 0.0 & 0.0 & 0.0 \\ 
		\hline \hline
	\end{tabular} \medskip
	\caption{\it Empirical rejection rates of tests \eqref{eq:test} and \eqref{eq:testLRV} for  the hypotheses \eqref{h0allspec}. The mean function is given by 
	\eqref{m2}, where   different values for the threshold  $\Delta$ and  different error processes are considered.
	 The variance is $\tilde{\sigma}_0^2(t)=1$,  the  sample size  is $n=500$ and the line in boldface corresponds to  the boundary of the hypotheses.  }
	\label{tab:comp2}
\end{table}}

\subsection{Case Study}

Time series with possibly smoothly varying mean naturally arise in the field of meteorology. To illustrate the proposed methodology, we consider the mean of daily minimal temperatures (in degrees Celsius) over the month of July for a period of approximately 120 years in eight different places in Australia.
At each station  we tested for relevant deviations of the temperature from    the mean temperature calculated for an historic reference period ranging from the late 19th century to 1925 at that station. As a threshold $\Delta$, we chose $0.25$, $0.5$ and $0.75$ degrees Celsius.
Exemplary, the observed temperature curves at the weather station in Cape Otway, Gayndah and Melbourne and the mean over all weather stations are plotted in Figure \ref{fig:realData}, alongside with their estimated smooth mean curves $\tilde \mu$ and the estimated benchmarks $\hat{g}$.

\begin{figure}[t!]
	\vspace{-.4cm}
\mbox{\hspace{-.5cm}
		\includegraphics[width=60mm]{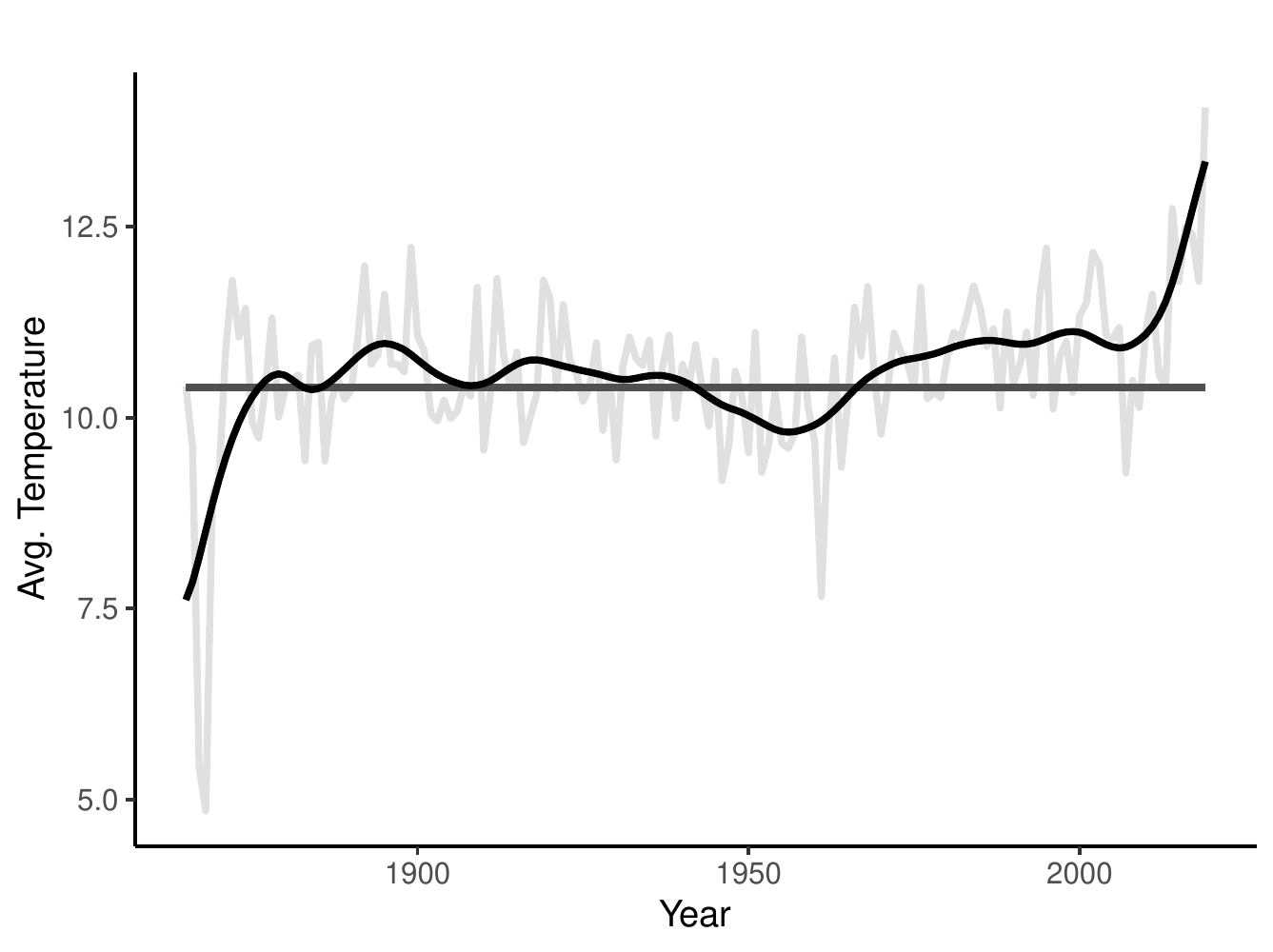}
		\includegraphics[width=60mm]{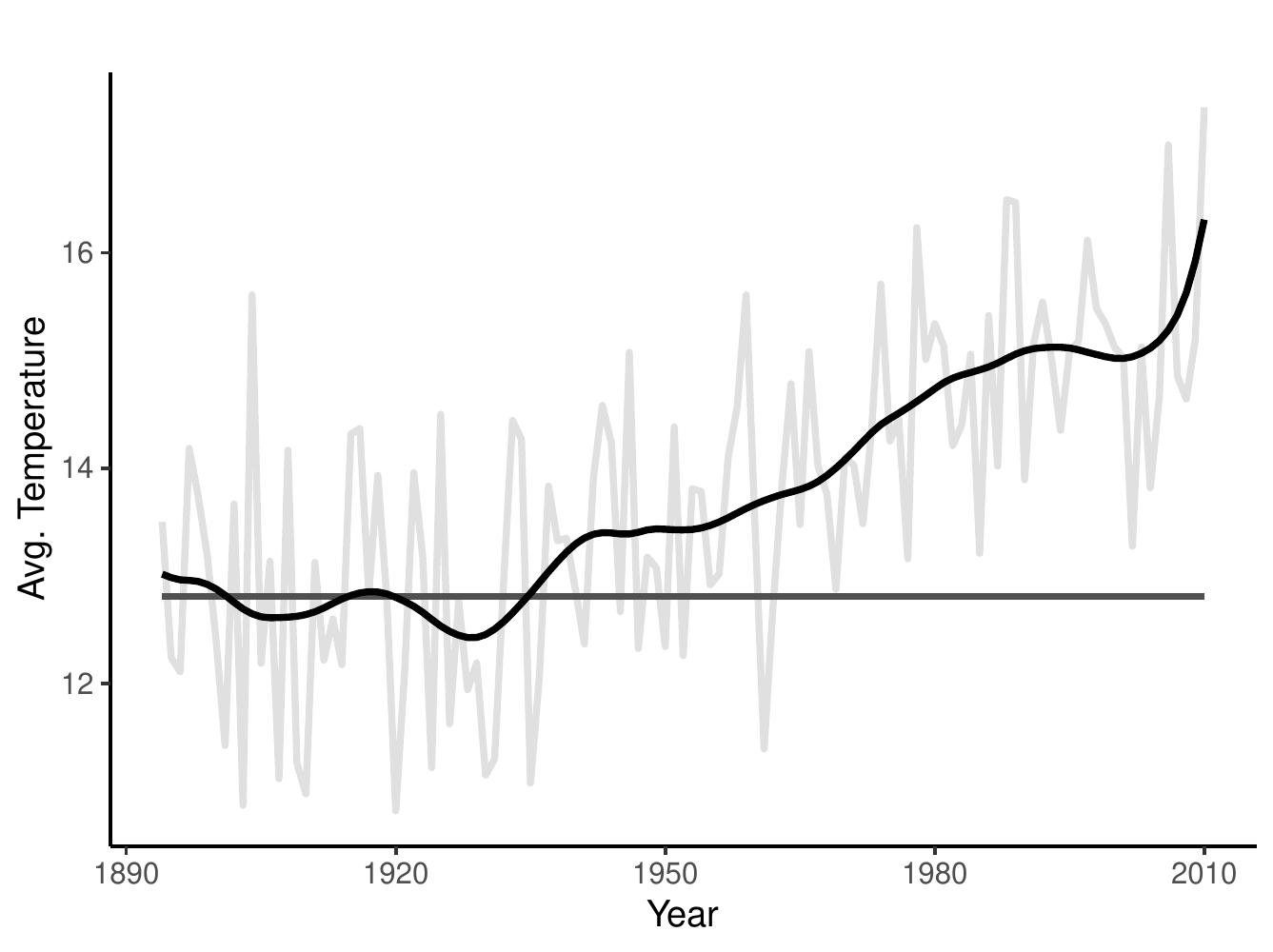}}  \\
	\vspace{-.28cm}
	\mbox{\hspace{-.5cm}
		\includegraphics[width=60mm]{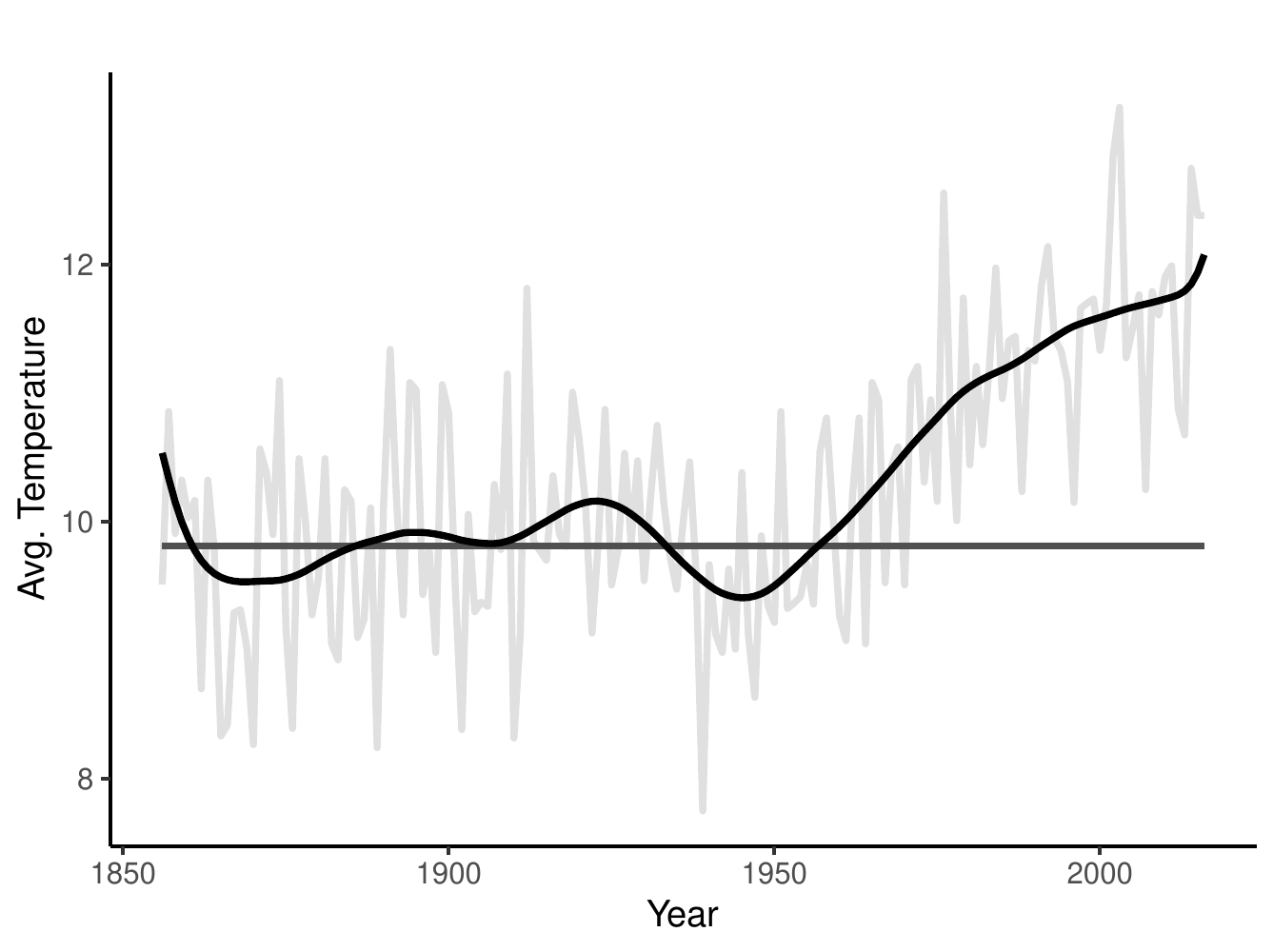}
		\includegraphics[width=60mm]{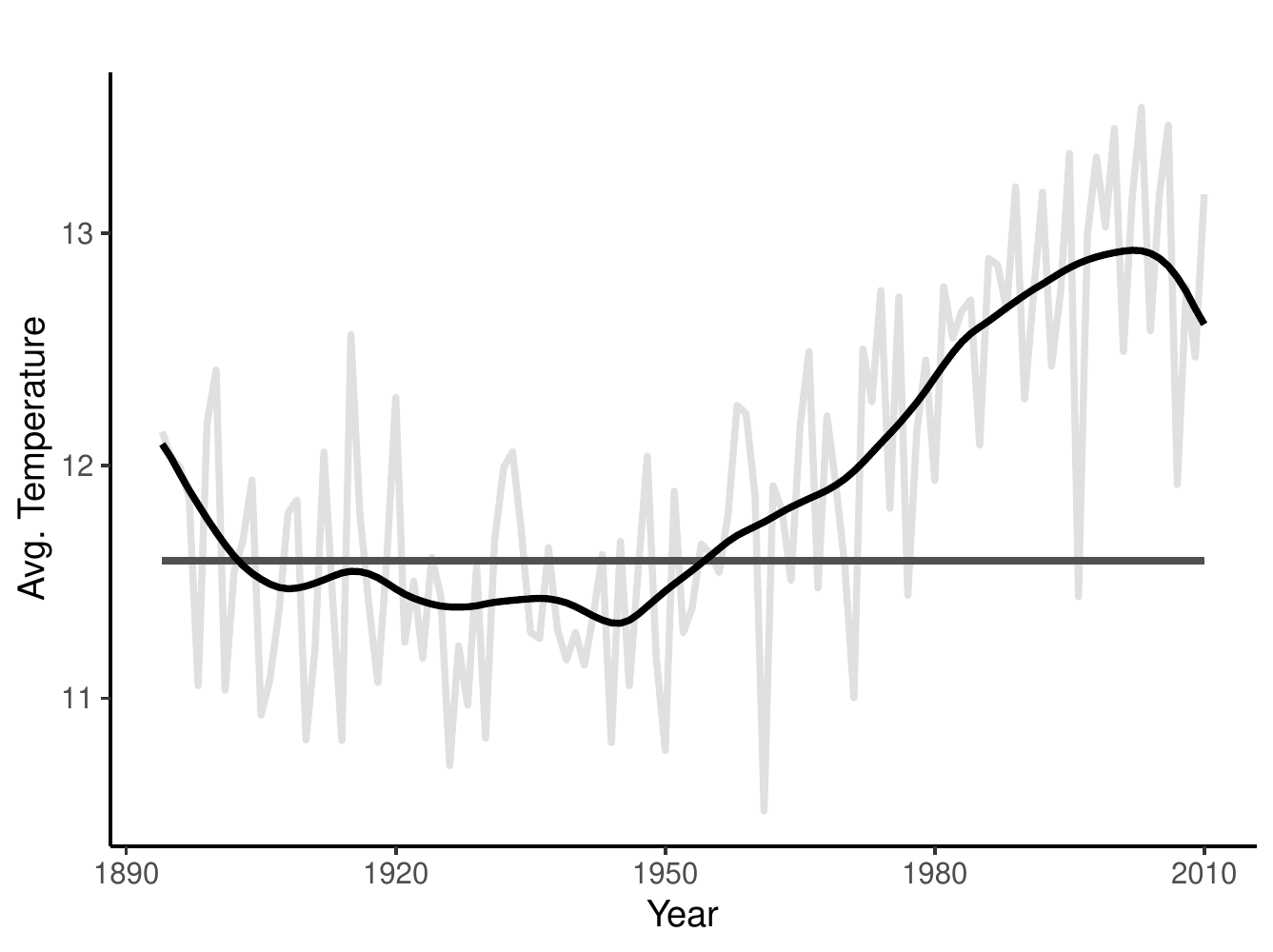}}
	\caption{\it{ Raw data of the temperature (light grey), estimated benchmark   (dark grey) and estimated smooth mean function for different weather stations. Top left: Cape Otway. Top right: Gayndah. Bottom left: Melbourne. Bottom right: Australia (mean). }}
	\label{fig:realData}
\end{figure}

The results for all stations under consideration can be found in Table~\ref{tab:app}. For test \eqref{eq:test}, most $p$-values are significant for $\Delta=0.25$  degrees Celsius. Further, two $p$-values for $\Delta = 0.5$ are significant. The test \eqref{eq:testLRV} does not yield a significant $p$-value below $0.05$ at any station. Test \eqref{eq:test} based on the proposed self-normalization procedure
seems to be more powerful than \eqref{eq:testLRV}, which confirms the numerical findings of the  simulation study.

{\small
	\begin{table}[t!]
		\begin{tabular}{r | r r r | r r r  } \hline \hline
			Test & \multicolumn{3}{c|}{\eqref{eq:test}} & \multicolumn{3}{c}{\eqref{eq:testLRV}}\\
			$\Delta$ & \multicolumn{1}{c}{0.25} & \multicolumn{1}{c}{0.5} & \multicolumn{1}{c|}{0.75} & \multicolumn{1}{c}{0.25} & \multicolumn{1}{c}{0.5} & \multicolumn{1}{c}{0.75} \\
			\hline
			Boulia Airport & \bf 4.2 & 7.8 & 24.2 & 22.1 & 28.5 & 40.8 \\
			Cape Otway Lighthouse & 7.3 & 99.3 & 100.0 & 41.6 & 70.1 & 96.1 \\
			Gayndah Post Office & \bf 0.7 & \bf 1.5 & 8.4 & 11.7 & 18.3 & 33.3 \\
			Gunnedah Pool & \bf 0.1 & \bf 0.5 & 11.8 & 13.5 & 22.2 & 41.9 \\
			Hobart & 5.9 & 53.7 & 98.6 & 25.5 & 51.1 & 87.9 \\
			Melbourne & \bf 2.3 & 59.1 & 99.7 & 29.5 & 51.5 & 84.1 \\
			Robe & 48.8 & 99.6 & 100.0 & 49.7 & 85.8 & 99.8 \\
			Sydney & \bf 2.3 & 98.4 & 100.0 & 39.0 & 62.6 & 90.6 \\
			Australia (mean) & \bf 0.4 & 99.8 & 100.0 & 37.4 & 65.5 & 94.5 \\
			\hline \hline
		\end{tabular} \medskip
		\caption{\it $p$-values of tests \eqref{eq:test} and \eqref{eq:testLRV} for the respective null hypotheses in percent. Significant $p$-values (below 0.05) are in boldface. }
		\label{tab:app}
	\end{table}
}

\medskip

\noindent {\bf Acknowledgements}
This work has been supported in part by the
Collaborative Research Center ``Statistical modeling of nonlinear
dynamic processes'' (SFB 823, Project A1,  C1) of the German Research Foundation (DFG).

\bibliographystyle{chicago}

\newpage

\appendix

\section{Proofs of main results} \label{secA}

In this section we will provide proofs of the theoretical statements in this paper. We begin with some preliminary results regarding the
uniform approximation of the  sequential estimators   of the regression function, which are of own interest and  required
for the proofs of the   main results in Section \ref{sec3}, which     will be given in Section \ref{sec:rem} and \ref{sec:proofs}.

\subsection{Sequential estimators of the regression function } \label{sec:details}

Recall the definition of the sequential Jackknife estimator $\tilde{\mu}_{h_n}(\lambda,t)$
in \eqref{eq:defJackknife} and define
\begin{align}
\label{hol7}
 K^*(x)=2\sqrt{2}K(\sqrt{2}x)-K(x)
\end{align}
 as the corresponding kernel.
  The following two results provide  stochastic expansions
for the difference $\tilde{\mu}_{h_n} - \mu$ uniformly with respect to $\lambda $ and $t$. Lemma \ref{thm:GaussianApprox} considers the case where the argument
$t$ stays away from the boundary, while a stochastic  expansion for the other case   is derived in Lemma \ref{thm:GaussianApproxRem} below.

\begin{lemma}\label{thm:GaussianApprox}
	Let $h_n\to 0$ and $b_n\to \infty$ be sequences with $b_n=o(nh_n)$ and define $I_n=[h_n,1-h_n]$.  If  Assumptions \ref{assumption:C2}, \ref{assump:kern} and \ref{assumption:LS} are satisfied,	we have
	\begin{equation*}
	\sup_{t\in I_n, \lambda\in [\zeta,1]} \Big| \lambda\big(\tilde{\mu}_{h_n}(\lambda,t)-\mu(t)\big)-\frac{1}{nh_n}\sum_{i=1}^{\lfloor \lambda n\rfloor}\eps_{T_i,n} K^*_{h_n}(T_i/n-t) \Big|=\mathcal{O}\big(h_n^3+\tfrac{b_n}{nh_n}\big),
	 \end{equation*}
where 	 $K^*_{h_n}  (\cdot ) = K^{*} ( \cdot /h_{n})$
	 and $K^* $ is defined in \eqref{hol7}
\end{lemma}

\begin{proof}
	Define
	$$ S_{n,j}(\lambda,t) = \sum_{i=1}^{\lfloor \lambda n\rfloor}\big(\tfrac{T_i-nt}{nh_n}\big)^j K_{h_n}\big(\tfrac{T_i}{n}-t\big) \quad \text{and} \quad R_{n,j}(\lambda,t) = \sum_{i=1}^{\lfloor \lambda n\rfloor}X_{T_i,n}\big(\tfrac{T_i-nt}{nh_n}\big)^j K_{h_n}\big(\tfrac{T_i}{n}-t\big),  $$
	for $j\in\{0,1,2\}$.
	Note that for  the calculation of the local linear estimator $\hat{\mu}_{h_n}(\lambda,t)$ in \eqref{eq:defLocLinEst}
	we have to minimize
	the function
	$$
	f(b_0,b_1)= \sum_{i=1}^{\lfloor \lambda n\rfloor} \big(X_{T_i,n}-b_0-b_1(T_i/n-t)\big)^2 K_{h_n}(T_i/n-t),
	$$
	which
	is differentiable with partial derivatives
	$$
	\frac{\partial f}{\partial b_j} (b_0,b_1)=- 2 h_n^j\big(R_{n,j}(\lambda,t)-b_0 S_{n,j}(\lambda,t)-b_1h_nS_{n,j+1}(\lambda,t)\big),$$
	for $j\in\{0,1\}$ and Hessian matrix
	\begin{equation}\label{eq:Hessian_matrix}
	\mathbf{H}_f = 2\left(\begin{array}{cc}
	S_{n,0}(\lambda,t) & h_nS_{n,1}(\lambda,t)  \\
	h_nS_{n,1}(\lambda,t) & h_n^2S_{n,2}(\lambda,t)
	\end{array}\right). \end{equation}
	In the following discussion we  will show that
	\begin{equation} \label{eq:approxSn} \sup_{\lambda\in[\zeta,1]}\bigg|\frac{1}{nh_n} S_{n,j}(\lambda,t) - \lambda \int_{-1}^1 x^j K(x)\diff x\bigg| =  \Oc\big(\tfrac{b_n}{nh_n}\big),\end{equation}
	for $j\in\{0,1,2\}$. If  this result is true, the proof follows by arguments similar to those used in the proof of Lemma B.1 of \cite{Dette2019}.
	To be precise, note that $$S_{n,0}(\lambda,t)S_{n,2}(\lambda,t)-S_{n,1}^2(\lambda,t)>0$$ for any $\lambda\in [\zeta,1]$ and almost every $n\in\N$. This means, that the
	Hessian matrix $\mathbf{H}_f$ is positive definite and both partial derivatives vanish if any only if
	\begin{equation}\label{eq:matrix_mult}
	\left(\begin{array}{cc}
	b_0 \\
	b_1
	\end{array}\right)
	=
	\left(\begin{array}{cc}
	S_{n,0}(\lambda,t) & h_nS_{n,1}(\lambda,t) \\
	S_{n,1}(\lambda,t) & h_nS_{n,2}(\lambda,t)
	\end{array} \right)^{-1}
	\left(\begin{array}{cc}
	R_{n,0}(\lambda,t) \\
	R_{n,1}(\lambda,t)
	\end{array}\right).
	\end{equation}
	By \eqref{eq:approxSn} and a Taylor expansion, it follows that
	\begin{multline*}
	\sup_{t\in I_n, \lambda\in [\zeta,1]} \Big| \lambda\big(\hat{\mu}_{h_n}(\lambda,t)-\mu(t)\big)-\frac{1}{nh_n}\sum_{i=1}^{\lfloor \lambda n\rfloor}\eps_{T_i,n} K_{h_n}(T_i/n-t)-\frac{\lambda}{2} h_n^2 \mu''(t)\int_{-1}^1 x^2K(x)\diff x \Big|\\=\mathcal{O}\big(h_n^3+\tfrac{b_n}{nh_n}\big). \end{multline*}
	The   statement of Lemma \ref{thm:GaussianApprox}  now  follows from the definition of the Jackknife estimator $\tilde{\mu}_{h_n}$ in \eqref{eq:defJackknife}.

	To finish the proof, we now  show the remaining estimate  \eqref{eq:approxSn}. For this purpose let  $\lambda\in[\zeta,1]$ and define  $B(\lambda)=\bigcup_{j=1}^{\ell_n} B_j(\lambda)$, where the sets $B_j(\lambda)$ are defined by
	$$
	B_j(\lambda)=\Big\{(j-1)b_n+1,\dots,(j-1)b_n+\lfloor\tfrac{\lambda n-1}{\ell_n}\rfloor +\id(j-1\leq \lfloor \lambda n\rfloor -1 \mod \ell_n) \Big\}.
	$$
	Note that (by Assumption \ref{assump:kern}) the kernel $K$ is Lipschitz continuous with support  $[-1,1]$, which implies
	$$K\big(\tfrac{i-nt}{nh_n}\big)=\left\{ \begin{array}{ll}
	K\big(\tfrac{jb_n-nt}{nh_n}\big)+\Oc\big(\tfrac{b_n}{nh_n}\big)&\text{if}~\big|i-nt\big|\leq nh_n,\\
	0&\text{else,}
	\end{array}\right.$$
	for $i\in B_j(\lambda)$, where the error term $\Oc(\tfrac{b_n}{nh_n})$ only depends on the function $K$ and, in particular, does not depend on $\lambda$. Thus, if follows that
	\begin{align*}
	S_{n,k}(\lambda,t)
	&=  \sum_{i=1}^{\lfloor \lambda n\rfloor} \big(\tfrac{T_i-nt}{nh_n}\big)^k K_{h_n}\big(\tfrac{T_i}{n}-t\big) =  \sum_{i\in B(\lambda)} \big(\tfrac{i-nt}{nh_n}\big)^j K_{h_n}\big(\tfrac{i}{n}-t\big)\\
	&=  \sum_{j=1}^{\ell_n} \sum_{i\in B_j(\lambda)}  \big(\tfrac{i-nt}{nh_n}\big)^k K_{h_n}\big(\tfrac{i}{n}-t\big)\\
	&=  \sum_{j=1}^{\ell_n} \Big( \lfloor \tfrac{\lambda n -1}{\ell_n}\rfloor + \id\big(j-1\leq \lfloor \lambda n\rfloor -1 \mod \ell_n\big) \Big) \big(\tfrac{jb_n-nt}{nh_n}\big)^k K_{h_n}\big(\tfrac{jb_n}{n}-t\big)+\Oc(b_n),
	\end{align*}
	uniformly in $\lambda$.	 As  the kernel $K$ has support  $[-1,1]$  the only non-zero summands in the last expression  are those with index $j\in[\tfrac{n(t-h_n)}{b_n},\tfrac{n(t+h_n)}{b_n}]$. Moreover, it holds that
	$$ \sup_{\lambda\in[\zeta,1]}\bigg|\frac{\lambda}{\ell_nh_n} \sum_{j=-\lfloor \ell_nh_n\rfloor}^{\lfloor \ell_nh_n \rfloor}\big(\tfrac{jb_n}{nh_n}\big)^j K\big(\tfrac{jb_n}{nh_n}\big)- \lambda \int_{-1}^1 x^j K(x)\diff x \bigg| = \Oc\big(\tfrac{b_n}{nh_n}\big), $$
	which implies \eqref{eq:approxSn}.
\end{proof}

\begin{lemma}\label{thm:GaussianApproxRem}
	Let $I_n^c = [0,1]\setminus I_n$, $\kappa_{j,h_n}(t)=\int_{-t/h_n}^{(1-t)/h_n} x^j K(x)\diff x$, for $j\in\N_0$, and $$c_{h_n,i}(t)=\frac{\kappa_{2,h_n}(t)-\kappa_{1,h_n}(t)\Big(\frac{T_i-nt}{nh_n}\Big)}{\kappa_{0,h_n}(t)\kappa_{2,h_n}(t)-\kappa_{1,h_n}^2(t)},$$
	for $i\in\{1,\dots,n\}$. If the assumptions of  Lemma \ref{thm:GaussianApprox} are satisfied, it holds
	\begin{multline}\label{eq:SNGaussianApproxRemainder}  \sup_{\substack{t\in I_n^c,\\ \lambda\in [\zeta,1]}} \Big| \lambda\big(\tilde{\mu}_{h_n}(\lambda,t)-\mu(t)\big)-\frac{1}{nh_n}\sum_{i=1}^{\lfloor \lambda n\rfloor}\eps_{T_i,n} \Big\{2c_{\frac{h_n}{\sqrt{2}},i}(t)K_{\frac{h_n}{\sqrt{2}}}\big(\tfrac{T_i}{n}-t\big)-c_{h_n,i}(t)K_{h_n}\big(\tfrac{T_i}{n}-t\big)\Big\} \Big|\\=\mathcal{O}\big(h_n^2+\tfrac{b_n}{nh_n}\big). \end{multline}
\end{lemma}

\begin{proof}
	By similar arguments as given for  the  approximation in \eqref{eq:approxSn} it follows that
	\begin{equation} \label{eq:approxSn2} \sup_{t\in I_n^c}\sup_{\lambda\in[\zeta,1]}\bigg|\frac{1}{nh_n} S_{n,j}(\lambda,t) - \lambda \kappa_{j,h_n}(t) \bigg| =  \Oc\big(\tfrac{b_n}{nh_n}\big),\end{equation}
	for $j\in\{0,1,2\}$. Note that $S_{n,0}(\lambda,t)S_{n,2}(\lambda,t)-S_{n,1}^2(\lambda,t)>0$ for any $\lambda\in [\zeta,1]$ and almost every $n\in\N$ since,
	by Assumption \ref{assump:kern},
	$$\int_\Ac K(x)\diff x \int_\Ac x^2K(x)\diff x - \bigg(\int_\Ac xK(x)\diff x\bigg)^2
	= \frac{1}{2}\int_{\Ac^2} (x-y)^2K(x)K(y)\diff(x,y),$$
	is positive for any  set $\Ac\subset[-1,1]$ with positive Lebesgue measure. Thus, the Hessian matrix $\mathbf{H}_f$, as defined in \eqref{eq:Hessian_matrix}, is positive definite and the partial derivatives vanish if and only if \eqref{eq:matrix_mult} holds true. Therefore, by \eqref{eq:approxSn2} and similar  arguments as given  in the proof of Lemma B.2 in \cite{Dette2019} we obtain that
	$$\sup_{t\in I_n^c, \lambda\in [\zeta,1]} \Big| \lambda\big(\hat{\mu}_{h_n}(\lambda,t)-\mu(t)\big)-\frac{1}{nh_n}\sum_{i=1}^{\lfloor \lambda n\rfloor}\eps_{T_i,n} c_{h_n,i}(t)K^*_{h_n}(T_i/n-t) \Big|=\mathcal{O}\big(h_n^2+\tfrac{b_n}{nh_n}\big)$$
	Finally the statement
	\eqref{eq:SNGaussianApproxRemainder} follows from the definition of the Jackknife estimator in \eqref{eq:defJackknife}.
	
\end{proof}

\subsection{Proof of Theorem \ref{thm:convSN}} \label{sec:proofs}

The following two Lemmas \ref{lem:fddConv} and \ref{lem:tightness} establish the convergence of the finite dimensional distributions and  equicontinuity of the process
$\{G_n(\lambda )\}_{\lambda \in [\zeta,1]} $ in \eqref{eq:Gn}
(for any $\zeta >0$). The assertion of Theorem \ref{thm:convSN} then follows directly from Theorems 1.5.4 and 1.5.7 of \cite{VanWel96}.

\begin{lemma}\label{lem:fddConv}
	Let Assumptions \ref{assumption:C2}, \ref{assump:kern}, \ref{assumption:LS}, \ref{assump:gn} and \ref{assump:seq} be satisfied and $\zeta \leq \lambda_1 \leq \dots \leq \lambda_p\leq 1$.
	Then,
	$$ \big( G_n(\lambda_1),\dots, G_n(\lambda_p) \big)^\top \convw \big( G(\lambda_1),\dots, G(\lambda_p) \big)^\top  $$
	in $\R^p$, where the Gaussian process $\{G(\lambda)\}_{\lambda \in [\xi,1]}$ is defined in \eqref{hol51}
\end{lemma}

\begin{proof}
	First observe that
	\begin{equation*}
	G_n(\lambda)= \lambda\sqrt{n}\big(\|\hat{d}_n(\lambda,\cdot)-d\|_{2,\tau}^2+ 2\langle d,\hat{d}_n(\lambda,\cdot)-d\rangle_\tau\big),
	\end{equation*}
	and note that 
	 \begin{equation}
	 \label{eq:g_hat-g}\hat{g}_n(\lambda)-g(\mu)=\Oc_\pr(n^{-1/2})
	 \end{equation}
	by Assumptions \ref{assumption:LS} and \ref{assump:gn}.
	Recall the definition of the interval $I_n=[h_n,1-h_n]$ and denote $\langle f,g \rangle_{I_n} = \int_{I_n} f(x)g(x)\tau(\diff x)$ and $\|f\|_{2,I_n} = \langle f,f\rangle_{I_n}^{1/2}$, for any $f,g\in L^2([0,1],\tau)$.
	In the following let 
	\begin{equation}\label{eq:approxD2}
	\tilde G_n(\lambda)= \lambda\sqrt{n}\big(\|\hat{d}_n(\lambda,\cdot)-d\|_{2,I_n}^2+ 2\langle d,\hat{d}_n(\lambda,\cdot)-d\rangle_{I_n}\big)~,
	\end{equation}
then the assertion follows from the statements
 \begin{align}\label{hol55}
    (\tilde G_n(\lambda_1), \ldots, \tilde G_n(\lambda_p))^\top  &\rightsquigarrow (G(\lambda_1),\ldots,G(\lambda_p))^\top
\\
\label{hol56}
     \sup_{\lambda \in [\zeta,1 ]}  | \tilde G_n(\lambda) - G_n(\lambda) |  &= o_\mathbb{P} (1)
 \end{align}
 For a proof of \eqref{hol55}  note that by the Cramér-Wold device, it is sufficient to prove 
	$$
	 \sum_{i=1}^{p} a_i  \tilde G_n(\lambda_i)\convw \sum_{i=1}^{p} a_i G(\lambda_i)
	 $$ for all  $a_1,\dots,a_p\in\R$.
 Define	
	 $S_k=\sum_{i=1}^{k}\eps_{i,n}$ and $\tilde S_k=\sum_{i=1}^{k}\tilde \eps_{i,n}$ with 	
		\begin{equation}
		\label{hol60} 
		\tilde{\eps}_{i,n}=\ex[\eps_{i,n}|\eta_i,\dots, \eta_{i-m_n}]
		\end{equation}
		 for $m_n\in\N$ and $k=1,\dots,n$.
	By Assumption \ref{assumption:LS} (1), Assumption \ref{assump:seq}
	 and equation (3.2) of \cite{Wu2011} we have
	\begin{equation}
	\label{eq:mDepApprox}
	\big\|\max_{1\leq k\leq n}|\tilde{S}_k-S_k| \big\|_{8,\Omega}=\Oc(n^{1/2}m_n^{-c}),
	\end{equation}
	where
	$m_n = n^\gamma $ with $  \gamma < {\min}\{ \beta/4 , \alpha/2\}$ and the constant $c$  is given by $c=\frac{5}{6\gamma}$. {In particular, the sequences $\frac{m_n^4}{nh_n^4}, \frac{m_n^2}{b_n}$ and $\frac{b_nn^{1/2}}{m_n^c}$ are all  of order $o(1)$ as $n$ tends to infinity.} From Lemma \ref{thm:GaussianApprox} it follows that
	\begin{equation}
	\label{eq:SNGaussianApprox2}
	\sup_{t\in [h_n,1-h_n], \lambda\in [\zeta,1]} \Big| \lambda\big(\tilde{\mu}_{h_n}(\lambda,t)-\mu(t)\big)-\frac{1}{nh_n}\sum_{i=1}^{\lfloor \lambda n\rfloor} \eps_{T_i,n} K^*_{h_n}(T_i/n-t) \Big|=\mathcal{O}\big(h_n^3+\tfrac{b_n}{nh_n}\big).
	\end{equation}	
	Observe that  by Assumption \ref{assump:gn}, \eqref{eq:g_hat-g}, \eqref{eq:mDepApprox} and \eqref{eq:SNGaussianApprox2}
	\begin{eqnarray}\label{hol52} \nonumber
	&&  \lambda\sqrt{n} \|\hat{d}_n(\lambda,\cdot)-d\|_{2,I_n}^2
	= \frac{\sqrt{n}}{\lambda} \bigg\| \bigg( \frac{1}{n} \sum_{j=1}^{\lfloor \lambda n\rfloor}  {\eps}_{T_j,n}\frac{1}{h_n} K_{h_n}^*(T_j/n-\cdot)\bigg) +  \lambda(\hat g_n(\lambda)-g(\mu))\bigg\|_{2,I_n}^2 + o_\pr(1) \\ \nonumber
& \leq & 2 \frac{\sqrt{n}}{\lambda} \bigg\| \frac{1}{n} \sum_{j=1}^{\lfloor \lambda n\rfloor} {\eps}_{T_j,n} \frac{1}{h_n}K_{h_n}^* (T_j/n-\cdot) \bigg\|^2_{2,I_n} + 2 \sqrt{n}\lambda (\hat g_n(\lambda) - g(\mu))^2 + o_\pr(1) \\ 
& = & 2  \frac{\sqrt{n}}{\lambda} \bigg\| \frac{1}{n} \sum_{j=1}^{\lfloor \lambda n\rfloor} {\tilde \eps}_{T_j,n} \frac{1}{h_n} K_{h_n}^* (T_j/n-\cdot) \bigg\|^2_{2,I_n} + o_\pr(1)
	 \end{eqnarray}
	 uniformly with respect to $\lambda\in [\zeta,1]$,
 where the random variables $\tilde \eps_{T_j,n}$ are  defined in \eqref{hol60} and $m_n$-dependent  in the sense that $\tilde  \eps_{T_{j_1},n}$ and $\tilde \eps_{T_{j_2},n}$ are independent  if $|T_{j_1}-T_{j_2}|>m_n$.
	For the estimation of the first term, let $K_{i,j}$ denote the integral $\int_{I_n} K_{h_n}^*(T_i/n-x)K_{h_n}^*(T_j/n-x)\tau(\diff x)$. By the Cauchy-Schwarz inequality and absolute continuity of $\tau$, $K_{i,j}$ can be bounded from above by $C h_n$. Tith implies
		\begin{align}\label{eq:vanishingTerm}
	\begin{split}
	&\ex\bigg[ \sup_{\lambda\in[\zeta, 1]} n^2 \bigg\| \frac{1}{nh_n} \sum_{j=1}^{\lfloor \lambda n\rfloor}\tilde\eps_{T_j,n} K_{h_n}^*(T_j/n-\cdot) \bigg\|_{2,I_n}^8\bigg]\\
	&\leq \sum_{k=1}^{n} \frac{n^2}{n^8h_n^8}\ex\bigg[\bigg\|\sum_{j=1}^{k}\tilde{\eps}_{T_j,n}K_{h_n}^*(T_j/n-\cdot)  \bigg\|_{2,I_n}^8\bigg]\\
	&= \sum_{k=1}^{n} \frac{1}{n^6h_n^8}\sum_{j_1,\dots,j_8=1}^k \ex\bigg[\prod_{r=1}^8\tilde{\eps}_{T_{j_r},n}\bigg] K_{j_1,j_2}K_{j_3,j_4}K_{j_5,j_6}K_{j_7,j_8} = \Oc(\tfrac{m_n^4}{nh_n^4}) ,
	\end{split}
	\end{align}
	where the last estimate follows observing  $\max_{1\leq i\leq n} \ex \eps_{i,n}^8 <\infty$ by Assumption \ref{assumption:LS} (4)
	and the fact that only  $\Oc(n^4m_n^4)$ summands of the inner sum are non-zero due 
	to the $m_n$-dependency of   the random variables $\tilde \eps_{j,n}$. Thus, by \eqref{hol52}, 
\begin{equation} \label{hol69}
\lambda\sqrt{n} \|\hat{d}_n(\lambda,\cdot)-d\|_{2,I_n}^2=o_\pr(1)
\end{equation}
	 uniformly with respect to $\lambda\in [\zeta,1]$.
		If 
		$$
		d_0=  \int_{0}^{1}d^2(x)dx =0,
		$$
		 it follows that	$ G_n(\lambda) = o_\pr(1) $
		and therefore we assume  
		$d_0>0$ in the following discussion. In this case 
		 we have from \eqref{eq:approxD2} and \eqref{eq:SNGaussianApprox2} that
	\begin{equation} \label{eq:CramerWold}
	\sum_{i=1}^{p}a_i \tilde G_n(\lambda_i) = 2 \sum_{i=1}^{p} a_i \lambda_i\sqrt{n}\langle d,\hat{d}_n(\lambda,\cdot)-d\rangle_{I_n} + o_\pr(1)=Z_n + o_\pr(1),
	\end{equation}
	where
	$$
	Z_n=  \frac{2}{\sqrt{n}}  \sum_{i=1}^{p}a_i\sum_{j=1}^{\lfloor \lambda_i n\rfloor}\tilde{\eps}_{T_j,n}\langle d,h_n^{-1} K_{h_n}^*(T_j/n-\cdot)+\omega_n(T_j/n)\rangle_{I_n}.
		$$
	By Lipschitz continuity of $d$ and $\supp(K)=[-1,1]$ it follows that
	$$
	\int_{I_n} d(x)K_{h_n}^*(T_j/n-x) \tau(\diff x)\\ = \Big(d\big(\tfrac{T_j}{n}\big)+\Oc(h_n)\Big)\int_{I_n} K_{h_n}^*(T_j/n-x)\tau(\diff x).
	$$
	We obtain  for any point of continuity $y$ of the  piecewise continuous density $f_\tau$ of the measure $\tau$ that 
	\begin{align*}
	\frac{1}{h_n}\int_{I_n} K_{h_n}^*(y-x)\tau(\diff x)
	&= \frac{1}{h_n}\int_{h_n}^{1-h_n} K_{h_n}^*(y-x)f_\tau(x) \diff x\\
	&= \int_{1-y/h_n}^{1/h_n-1-y/h_n} K^*(x) f_\tau(xh_n+y)\diff x\\
	& = \left\{\begin{array}{ll}
	f_\tau(y)+o(1),&\text{if}~y\in[2h_n,1-2h_n],\\
	\Oc(1),&\text{else}.
	\end{array} \right.\end{align*}
	Therefore, for $T_j\in \{2nh_n,\dots,n-2nh_n\}$, it holds
	\begin{equation}\label{eq:scalarProdMuK}
	\frac{1}{h_n}\int_{I_n} d(x)K_{h_n}^*(T_j/n-x) \tau(\diff x)
	= f_\tau(T_j/n)d(T_j/n) + o(1),	\end{equation}
	which leads to
	\begin{align} \label{hol62}
	Z_n
	&= \sum_{j=1}^{n}  Y_j +o_\pr(1)
	\end{align}
	where the random variables $Y_1,\dots, Y_n$ are defined by  
		$$ Y_j= \frac{2}{\sqrt{n}}\bigg(\sum_{i=1}^p    a_i \id\big(j\leq \lfloor \lambda_i n\rfloor\big)\bigg)\tilde{\eps}_{T_j,n}d_{\omega_n}(T_j/n) \qquad (j=1,\dots,n) $$
	and
	$$
	d_{\omega_n}(T_j/n) = f_\tau(T_j/n)d(T_j/n) + \omega_n(T_j/n) \int_0^1 d(x)\tau(\diff x) .
	$$
	Observe that $Y_1,\dots, Y_n$ centred and $m_n$-dependent   random variables in the sense that $ Y_{j_1}$ and $Y_{j_2}$ are independent  if $|T_{j_1}-T_{j_2}|>m_n$. Define the big blocks $ B_j= \{ k\in\N: (j-1)b_n+1\leq k\leq jb_n-m_n  \}$ and the small blocks $S_j=\{ k\in\N: jb_n-m_n+1\leq k\leq jb_n  \}$, for $j=1\dots, \ell_n$, and the remainder $R=\{ k\in\N: \ell_n b_n+1\leq k \leq n \}$. In the following, we will show that the small blocks and the remainder are negligible and the asymptotic behaviour of $Z_n$ is determined by  the big blocks.
	First observe that $\tilde{\eps}_{T_{k_1},n}\in S_{j_1}$ and $\tilde{\eps}_{T_{k_2},n}\in S_{j_2}$ are independent for $j_1\neq j_2$. Thus,
	\begin{align}\label{eq:smallBlocks}\begin{split}
	&\ex\bigg[\bigg(\sum_{j=1}^{\ell_n}\sum_{k:T_k\in S_j}Y_k\bigg)^2\bigg]\\
	&= \sum_{j=1}^{\ell_n} \sum_{k_1: T_{k_1} \in S_{j}}\sum_{k_2: T_{k_2} \in S_{j}} \ex[ Y_{k_1}Y_{k_2}]\\
	&= \frac{4}{n}\sum_{j=1}^{\ell_n} \sum_{i_1,i_2=1}^{p}  a_{i_1}a_{i_2} \sum_{k_1=1}^{\lfloor \lambda_{i_1}n\rfloor} \sum_{k_2=1}^{\lfloor \lambda_{i_2} n\rfloor}\id(T_{k_1},T_{k_2}\in S_{j})\ex[\tilde{\eps}_{T_{k_1},n}\tilde{\eps}_{T_{k_2},n}]d_{\omega_n}(T_{k_1}/n)d_{\omega_n}(T_{k_2}/n).
	\end{split}
	\end{align}
	Further, $T_k\in S_j$ for some $k\leq \lfloor\lambda n\rfloor$, if and only if $k=r\ell_n+j$ for some $r\geq b_n-m_n $ and $r\leq \lfloor \tfrac{\lambda n-j}{\ell_n}\rfloor$ and we obtain
	\begin{align*}
	\begin{split}
	&\ex\bigg[\bigg(\sum_{j=1}^{\ell_n}\sum_{k:T_k\in S_j}Y_k\bigg)^2\bigg]\\
	&=  \sum_{i_1,i_2=1}^{p}  a_{i_1}a_{i_2} \frac{4}{n} \sum_{j=1}^{\ell_n}\sum_{r_1= b_n-m_n}^{\lfloor \tfrac{\lambda_{i_1}n-j}{\ell_n}\rfloor}\sum_{r_2=b_n-m_n}^{\lfloor \tfrac{\lambda_{i_2}n-j}{\ell_n}\rfloor} d_{\omega_n}\big(\tfrac{(j-1)b_n+r_1+1}{n}\big)  d_{\omega_n}\big(\tfrac{(j-1)b_n+r_2+1}{n}\big)\\ & \hspace{1cm}\times \ex[\tilde{\eps}_{(j-1)b_n+r_1+1,n}\tilde{\eps}_{(j-1)b_n+r_2+1,n}].
	\end{split}
	\end{align*}	
	For $\lambda < 1$, it holds that $\lambda \tfrac{n}{b_n\ell_n}\to \lambda$ and $1-\tfrac{m_n}{b_n}\to 1$, so for almost every $n\in\N$, $b_n-m_n\geq  \lfloor \tfrac{\lambda n-j}{\ell_n}\rfloor$. Thus, if $\lambda_{i_1}<1$ or $\lambda_{i_2}<1$, the sums indexed by $r_1$ and $r_2$ on the right-hand side of the previous display are empty sums for almost every $n\in\N$. For $\lambda=1$, there are $m_n$ summands in both sums, thus, the right-hand side of the previous display is of order $\Oc(m_n^2/b_n)$ which vanishes by assumption. Thus the small blocks are asymptotically negligible, and analogously,
	$$ \ex\bigg[\bigg(\sum_{k\in \bar{R}} Y_k\bigg)^2\bigg]=\Oc\Big(\frac{b_n^2}{n}\Big). $$
\\	
	The sums over the big blocks are independent, and we have analogously to \eqref{eq:smallBlocks},
	$$\sum_{j=1}^{\ell_n}\ex\bigg[\bigg(\sum_{k:T_k\in B_j}Y_k\bigg)^2\bigg]
	= \sum_{i_1,i_2=1}^{p}  a_{i_1}a_{i_2} \frac{4}{n} \sum_{j=1}^{\ell_n}\sum_{r_1\in \bar B_{i_1,j}}\sum_{r_2\in \bar B_{i_2,j}}d_{\omega_n}\big(\tfrac{r_1}{n}\big)  d_{\omega_n}\big(\tfrac{r_2}{n}\big) \ex[\tilde{\eps}_{r_1,n}\tilde{\eps}_{r_2,n}]$$
	where $$\bar B_{i,j} = \{(j-1)b_n+1, \dots, (j-1)b_n+1+\lfloor \tfrac{\lambda_{i}n-j}{\ell_n}\rfloor \wedge (b_n-m_n-1) \},$$
	for $\lambda_i\in [\zeta,1]$. Note that, for almost every $n\in\N$, $\bar B_{i,j} = \{(j-1)b_n+1, \dots, (j-1)b_n+1+\lfloor \tfrac{\lambda_{i}n-j}{\ell_n}\rfloor \}$, if $\lambda<1$ and $\bar B_{i,j} = \{(j-1)b_n+1, \dots, jb_n-m_n \}$, if $\lambda=1$. By Lipschitz continuity of $d$ and Assumption \ref{assump:gn},
	\begin{equation}\label{eq:lemmaFDDconv1}
	\sum_{j=1}^{\ell_n}\ex\bigg[\bigg(\sum_{k:T_k\in B_j}Y_k\bigg)^2\bigg]
	=  \sum_{i_1,i_2=1}^{p}  a_{i_1}a_{i_2}  \frac{4}{n} \sum_{j=1}^{\ell_n} d_{\omega_n}^2\big(\tfrac{jb_n}{n}\big) \sum_{r_1\in \bar B_{i_1,j}}\sum_{r_2\in \bar B_{i_2,j}} \ex[\tilde{\eps}_{r_1,n}\tilde{\eps}_{r_2,n}]+\Oc\big(\tfrac{b_n^2}{nh_n}\big).\\
	\end{equation}		
	Now, by \eqref{eq:mDepApprox},
	\begin{equation}\label{eq:mDepCov}
	\max_{1\le  r_1, r_2\le n}\big|\ex[\tilde{\eps}_{ r_1,n}\tilde{\eps}_{ r_2,n}]-\ex[\eps_{ r_1,n}\eps_{ r_2,n}]\big|=\Oc(n^{1/2}m_n^{-c}).
	\end{equation}
	Applying Assumption \ref{assumption:LS} (2), yields $$\ex[\eps_{ r_1,n}\eps_{ r_2,n}] = \ex\big[G\big(\tfrac{jb_n}{n},\Fc_{ r_1}\big)G\big(\tfrac{jb_n}{n},\Fc_{ r_2}\big)\big]+\Oc(b_n/n),$$
	for any $ r_1 \in \bar B_{i_1,j}, r_2\in \bar B_{i_2,j}$. By the same arguments as in the proof of Theorem 1 in \cite{Wu2009} and Assumption \ref{assumption:LS} (1), it follows that
	\begin{equation*}
	\ex\big[G\big(\tfrac{i}{n},\Fc_{ r_1}\big)G\big(\tfrac{i}{n},\Fc_{ r_2}\big)\big]=\Oc(\gamma^{| r_2- r_1|}),
	\end{equation*}
	for any $1\leq i \leq n$, in particular $i=jb_n$. Let $ b := |\bar B_{i_1,j}\cap \bar B_{i_2,j}|$, then, 
	\begin{align}\label{eq:lrvApprox}
	\begin{split}
	\sum_{ r_1\in \bar B_{i_1,j}}\sum_{ r_2\in \bar B_{i_2,j}}\ex[\eps_{ r_1,n}\eps_{ r_2,n}]
	&=  b \sum_{k=- b}^{ b} \big(1-\tfrac{|k|}{ b}\big) \ex\big[G\big(\tfrac{jb_n}{n},\Fc_{0}\big)G\big(\tfrac{jb_n}{n},\Fc_{k}\big)\big]+ \Oc(b_n^3/n+b_n\gamma^{b_n}+1)\\
	&= (\lambda_{i_1}\wedge \lambda_{i_2})b_n \sigma^2\bigg(\frac{jb_n}{n}\bigg) + \Oc(b_n^3/n+b_n\gamma^{b_n}+1).
	\end{split}
	\end{align}
	Thus, by \eqref{eq:mDepCov},
	\begin{multline*}
	\sum_{ r_1\in \bar B_{i_1,j}}\sum_{ r_2\in \bar B_{i_2,j}}\ex[\tilde \eps_{ r_1,n}\tilde \eps_{ r_2,n}]=  \big(\lambda_{i_1}\wedge \lambda_{i_2}\big)b_n \sigma^2\bigg(\frac{jb_n}{n}\bigg)+\Oc\bigg(\frac{b_n^3}{n}+b_n\gamma^{b_n}+1+\frac{b_n^2n^{1/2}}{m_n^{c}}\bigg).
	\end{multline*}
Plugging this into \eqref{eq:lemmaFDDconv1} and observing $\ell_n=\lfloor n/b_n\rfloor$  leads to
	\begin{align}\label{eq:Riemann_conv}
	\sum_{j=1}^{\ell_n}\ex\bigg[\bigg(\sum_{k:T_k\in B_j}Y_k\bigg)^2\bigg] & =    \sum_{i_1,i_2=1}^{p} 4a_{i_1}a_{i_2}(\lambda_{i_1}\wedge \lambda_{i_2})\|d_\omega\sigma\|_{2}^2 +  o(1)
	\\
	& =\var\bigg(\sum_{i=1}^{p}a_iG(\lambda_i)\bigg) + o(1). \nonumber 
	\end{align}
	Finally, observe that by Jensen's inequality and Assumption \ref{assumption:LS} (4), for some constant $C>0$,
	$$
	\sum_{j=1}^{\ell_n} \ex\bigg[\bigg(\sum_{k: T_k\in B_j} Y_k\bigg)^4\bigg]
	\leq \sum_{j=1}^{\ell_n} b_n^3\ex\bigg[\sum_{k: T_k\in B_j} Y_k^4\bigg]
	\leq C\frac{b_n^3}{n}\max_{k=1}^n \ex\, \tilde{\eps}_{T_k,n}^4
	= \Oc(b_n^3/n).
	$$
	By Lyapunov's central limit theorem, it follows that
	$$ Z_n \convw \Nc\bigg(0, \var\bigg(\sum_{i=1}^{p}a_iG(\lambda_i)\bigg)\bigg)\stackrel{\Dc}{=}\sum_{i=1}^{p}a_i G(\lambda_i) $$
	and the   statement  \eqref{hol55}  is a consequence of \eqref{eq:CramerWold}, \eqref{hol62} and   the Cramér-Wold device.
\\	
	For the proof of the  remaining statement \eqref{hol56} we note that this assertion is a consequence of the estimate
	\begin{equation} \label{hol57}
\sup_{\lambda\in[\zeta,1]} \int_{I_n^c} \lambda \sqrt{n} \big(\tilde{\mu}_{h_n}(\lambda,t)-\mu(t)\big)^j\diff t   = o_\pr(1) \qquad j\in\{1,2\}.
\end{equation}
To prove this statement, we note that by Lemma \ref{thm:GaussianApproxRem} 
	\begin{multline*}
	\sup_{\lambda\in[\zeta,1]} \int_{I_n^c} \lambda \sqrt{n} \big(\tilde{\mu}_{h_n}(\lambda,t)-\mu(t)\big)^j\diff t   =\\ \sup_{\lambda\in[\zeta,1]} \int_{I_n^c}  \frac{1}{\lambda^{j-1}}\sqrt{n}\bigg(\frac{1}{nh_n}\sum_{i=1}^{\lfloor \lambda n\rfloor}\eps_{T_i,n}\Big( 2c_{\frac{h_n}{\sqrt{2}},i}(t)K_{\frac{h_n}{\sqrt{2}}}\big(\tfrac{T_i}{n}-t\big)-c_{h_n,i}(t)K_{h_n}\big(\tfrac{T_i}{n}-t\big)\Big) \bigg)^j\diff t  + o(1),\end{multline*}
	for $j\in\{1,2\}$. The case $j=2$ follows by similar arguments as given in   \eqref{eq:vanishingTerm}. For the case $j=1$ recall from the previous discussion    that the random variables $\eps_{i,n}$ can be approximated by $m_n$-dependent random variables $\tilde{\eps}_{i,n}$. Thus,
	\begin{align}\label{eq:cor_rem}
	\begin{split}
	&\sup_{\lambda\in[\zeta,1]} \int_{I_n^c} \lambda \sqrt{n} \big(\tilde{\mu}_{h_n}(\lambda,t)-\mu(t)\big)\diff t   \\
	&=\sup_{\lambda\in[\zeta,1]} \int_{I_n^c} \frac{1}{\sqrt{n}h_n}\sum_{i=1}^{\lfloor \lambda n\rfloor}\tilde\eps_{T_i,n}\Big( 2c_{\frac{h_n}{\sqrt{2}},i}(t)K_{\frac{h_n}{\sqrt{2}}}\big(\tfrac{T_i}{n}-t\big)-c_{h_n,i}(t)K_{h_n}\big(\tfrac{T_i}{n}-t\big)\Big)\diff t  + o_\pr(1)\\
	&= \sup_{\lambda\in[\zeta,1]} \frac{1}{\sqrt{n}}\sum_{i=1}^{\lfloor \lambda n\rfloor}\tilde\eps_{T_i,n}  \int_{I_n^c} \frac{1}{h_n} \Big(2c_{\frac{h_n}{\sqrt{2}},i}(t)K_{\frac{h_n}{\sqrt{2}}}\big(\tfrac{T_i}{n}-t\big)-c_{h_n,i}(t)K_{h_n}\big(\tfrac{T_i}{n}-t\big)\Big)\diff t  + o_\pr(1)\\
	&= \sup_{\lambda\in[\zeta,1]} \frac{1}{\sqrt{n}}\sum_{j\in B}\sum_{i=1}^{\lfloor \lambda b_n\rfloor}\tilde\eps_{(j-1)b_n+i,n} \\
	&\hspace{1cm}\times \int_{I_n^c} \frac{1}{h_n} \Big( 2c_{\frac{h_n}{\sqrt{2}},i}(t)K_{\frac{h_n}{\sqrt{2}}}\big(\tfrac{(j-1)b_n+i}{n}-t\big)-c_{h_n,i}(t)K_{h_n}\big(\tfrac{(j-1)b_n+i}{n}-t\big)\Big)\diff t  + o_\pr(1),
	\end{split}\end{align}
	where $B$ denotes the set $\{ 1,\dots, \lfloor 2\ell_nh_n\rfloor \}\cup \{ \lfloor \ell_n(1-2h_n)\rfloor,\dots, \ell_n \}$. Note that the integral on the right-hand side of  \eqref{eq:cor_rem} is bounded and by 
	similar  arguments as used in the proof of \eqref{eq:vanishingTerm}, the right-hand side of \eqref{eq:cor_rem} is of order $\Oc(b_nh_n^4m_n^4)$, which converges to $0$ by the definition of $m_n$ and Assumption \ref{assump:seq}.

	\medskip
	
	Therefore \eqref{hol57} follows and the proof of Lemma \ref{lem:fddConv} is completed.
\end{proof}

	\medskip

\begin{lemma}\label{lem:tightness}
	Let Assumptions \ref{assumption:C2}, \ref{assump:kern}, \ref{assumption:LS}, \ref{assump:gn} and \ref{assump:seq} be satisfied. Then,
	$$ \lim\limits_{\rho \searrow 0}\lim\limits_{n\to\infty}\pr(\sup_{|\lambda_1-\lambda_2|\leq \rho} |G_n(\lambda_1)-G_n(\lambda_2)|>\eps) =0, $$ 
	for any $\eps>0$.
\end{lemma}

\begin{proof} 
	By \eqref{hol56}, it follows that
		\begin{align*}
	G_n(\lambda_1)-G_n(\lambda_2)
	&=  \tilde 	G_n(\lambda_1)- \tilde  G_n(\lambda_2)
	+o_\pr(1).
	\end{align*}
	uniformly with respect to $\lambda \in [\zeta, 1]$, where $\tilde 	G_n(\lambda )$ is defined in 
	\eqref{eq:approxD2}. Therefore the assertion of the Lemma follows from
	\begin{equation} \label{hol68} 
	 \lim\limits_{\rho \searrow 0}\lim\limits_{n\to\infty}\pr(\sup_{|\lambda_1-\lambda_2|\leq \rho} | \tilde G_n(\lambda_1)- \tilde G_n(\lambda_2)|>\eps) =0, 
	\end{equation}
To prove this statement note that we obtain  from \eqref{hol69} 
	$$
 \tilde 	G_n(\lambda_1)- \tilde G_n(\lambda_2)
	=  2\sqrt{n}\langle \lambda_1\big(\hat{d}_n(\lambda_1,\cdot)-d\big)-\lambda_2\big(\hat{d}_n(\lambda_2,\cdot)-d\big),d \rangle_{I_n}+o_\pr(1)
	$$
	uniformly with respect to   $\lambda_1, \lambda_2 \in [\zeta, 1]$.
	By Lemma \ref{thm:GaussianApprox},  Assumption \ref{assump:gn} and  \eqref{eq:scalarProdMuK} we have the expansion
	\begin{align*}
	 \tilde G_n(\lambda_1)- \tilde  G_n(\lambda_2)
	& =  \frac{2}{\sqrt{n}}\sum_{i=\lfloor (\lambda_1\wedge\lambda_2)n\rfloor+1}^{\lfloor (\lambda_1\vee\lambda_2)n\rfloor} \eps_{T_i,n}\bigg\langle \frac{1}{h_n}K_{h_n}^*\bigg(\frac{T_i}{n}-\cdot\bigg)+\omega_n\bigg(\frac{T_i}{n}\bigg),d \bigg\rangle_{I_n}
	+o_\pr(1) \\
	&=  \frac{2}{\sqrt{n}}\sum_{i=\lfloor (\lambda_1\wedge\lambda_2)n\rfloor+1}^{\lfloor (\lambda_1\vee\lambda_2)n\rfloor} \eps_{T_i,n} d_{\omega_n}(T_i/n)	+o_\pr(1),
	\end{align*}
	uniformly in $\lambda_1, \lambda_2\in [\zeta, 1] $, 
	where  $d_{\omega_n}$ is  defined in the proof of Lemma \ref{lem:fddConv}.	
	Further, recalling the definition of   $\tilde{\eps}_{i,n}$  in  \eqref{hol60}, it follows 
	 by \eqref{eq:mDepApprox}, that 
	$$
	\tilde 
	G_n(\lambda_1)- \tilde 
	G_n(\lambda_2)
	=  \frac{2}{\sqrt{n}}\sum_{i=\lfloor (\lambda_1\wedge\lambda_2)n\rfloor+1}^{\lfloor (\lambda_1\vee\lambda_2)n\rfloor} \tilde \eps_{T_i,n}d_{\omega_n}(T_i/n)+o_\pr(1)$$
	uniformly in $\lambda_1, \lambda_2\in [\zeta, 1] $.	In particular, we obtain 
	\begin{multline}\label{eq:lemTightness2}
	\pr\Big(\sup_{|\lambda_1-\lambda_2|\leq \rho} | \tilde  G_n(\lambda_1)- \tilde G_n(\lambda_2)|>\eps \Big)\\
	= \pr\bigg(\sup_{|\lambda_1-\lambda_2|\leq \rho} \bigg|\frac{2}{\sqrt{n}}\sum_{i=\lfloor (\lambda_1\wedge\lambda_2)n\rfloor+1}^{\lfloor (\lambda_1\vee\lambda_2)n\rfloor} \tilde \eps_{T_i,n}d_{\omega_n}(T_i/n) \bigg| >\eps \bigg)+o(1).
	\end{multline}

	Now, for some $\zeta \leq \lambda_1\leq \lambda_2 \leq 1 $,  define  the sets $\tilde{B}_j=\tilde{B}_j(\lambda_1,\lambda_2)$ by  
	\begin{multline*}
	\bigg\{i\in \N : (j-1)b_n + \Big\lfloor \frac{\lambda_1 n}{\ell_n}\Big\rfloor + \id(j < \lfloor \lambda_1 n\rfloor\mod \ell_n)+1\\\leq i \leq (j-1)b_n+ \Big\lfloor \frac{\lambda_2 n}{\ell_n}\Big\rfloor - \id(\lfloor \lambda_2 n\rfloor\mod\ell_n <j) +1 \bigg\},\end{multline*}
	for $j=1,\dots,\ell_n$. In particular, $|\tilde{B}_j| \leq \lfloor \tfrac{\lambda_2n}{\ell_n}\rfloor - \lfloor \tfrac{\lambda_1n}{\ell_n}\rfloor+1\leq \lfloor \tfrac{|\lambda_2-\lambda_1|n}{\ell_n}\rfloor+2$. With this notation, it holds
	\begin{equation*} 
	R_n=
	\ex\bigg[\bigg|\frac{1}{\sqrt{n}}\sum_{i=\lfloor (\lambda_1\wedge\lambda_2)n\rfloor+1}^{\lfloor (\lambda_1\vee\lambda_2)n\rfloor} \tilde \eps_{T_i,n}d_{\omega_n}(T_i/n) \bigg|^4\bigg] = \frac{1}{n^2} \ex\bigg[\bigg| \sum_{j=1}^{\ell_n}\sum_{i \in \tilde B_j} \tilde{\eps}_{i,n}d_{\omega_n}\big(i/n\big)\bigg|^4 \bigg].
	\end{equation*}
	Observe that the distance between two blocks $\tilde{B}_j$ and $\tilde{B}_{j-1}$ is larger than 
	$b_n - \Big\lfloor \frac{(\lambda_2-\lambda_1) n}{\ell_n}\Big\rfloor > m_n$. Thus, the sums over these blocks are independent and we obtain the
	representation
	\begin{align}\label{eq:lemTightness}
	\begin{split}
R_n=	&\frac{1}{n^2} \sum_{j=1}^{\ell_n} \sum_{i_1,i_2,i_3,i_4 \in \tilde B_j} \bigg(\prod_{r=1}^4 d_{\omega_n}\big(i_r/n\big)\bigg) \ex\Big[\prod_{r=1}^4 \tilde{\eps}_{i_r,n} \Big]\\
	&+\frac{3}{n^2} \sum_{\substack{j_1,j_2=1\\j_1\neq j_2}}^{\ell_n} \sum_{i_1,i_2 \in \tilde B_{j_1}}\sum_{i_3,i_4\in \tilde{B}_{j_2}} \bigg(\prod_{r=1}^4 d_{\omega_n}\big(i_r/n\big)\bigg) \ex[\tilde{\eps}_{i_1,n} \tilde{\eps}_{i_2,n}]\ex[\tilde{\eps}_{i_3,n} \tilde{\eps}_{i_4,n}]. 
	\end{split}
	\end{align}
	We first consider the first term of \eqref{eq:lemTightness}. Recall that the random variables $\tilde{\eps}_{i,n}$ are $m_n$-dependent and that $|\tilde{B}_j|\leq\lfloor \tfrac{|\lambda_1-\lambda_2|n}{\ell_n}\rfloor+2$. Therefore, the number of non-zero summands in the inner sum can be bounded from above by $\big(\lfloor \tfrac{|\lambda_1-\lambda_2|n}{\ell_n}\rfloor+2\big)^2 m_n^2$. By Assumption \ref{assumption:LS} (4), Assumption \ref{assump:gn} and Lipschitz continuity of $d$, the first term in \eqref{eq:lemTightness} can be bounded from above by $C|\lambda_1-\lambda_2|^2 \tfrac{m_n^2}{\ell_n}\leq C|\lambda_1-\lambda_2|^2$.
	
	The second term of \eqref{eq:lemTightness} is  bounded  by
	$$ 
	S_n = C\bigg(\frac{1}{n}\sum_{j=1}^{\ell_n}  \sum_{i_1,i_2 \in \tilde B_j} d_{\omega_n}\big(i_1/n\big)d_{\omega_n}\big(i_2/n\big) \ex[\tilde{\eps}_{i_1,n} \tilde{\eps}_{i_2,n}]\bigg)^2
	$$
	for some constant $C$. 
	The inner sum in this term can be rewritten as $\tfrac{|\lambda_1-\lambda_2|n}{\ell_n} d_{\omega_n}^2\big(\tfrac{jb_n}{n}\big)\sigma^2\big(\tfrac{jb_n}{n}\big)+o(1)$, analogously to \eqref{eq:lrvApprox}. Thus, $S_n$  converges to
	$ C|\lambda_1-\lambda_2|^2 \|d_\omega\sigma\|_2^4$. Combining these arguments we obtain  
$$
R_n \leq C |\lambda_1-\lambda_2|^2.
$$ 	
	Therefore, by Theorem 2.2.4 of \cite{VanWel96} it follows that
	\begin{multline*}
	\ex\bigg[\sup_{|\lambda_1-\lambda_2|\leq \rho} \bigg|\frac{1}{\sqrt{n}}\sum_{i=\lfloor (\lambda_1\wedge\lambda_2)n\rfloor+1}^{\lfloor (\lambda_1\vee\lambda_2)n\rfloor} \tilde \eps_{T_i,n}d_{\omega_n}(T_j/n) \bigg|^4\bigg]\\
	\leq K^4\bigg\{ \int_0^\eta D^{1/4}(\eps)\diff \eps + \rho^{1/2} D^{1/2}(\eta) \bigg\}^4
	\leq K^4\Big(2\eta^{1/2}+\frac{\rho^{1/2}}{\eta}\Big)^4,
	\end{multline*}
	for some constant $K$ and any $\eta>0$, where $D(\eps)$ denotes the packing number of the space $([0,1],|\cdot|^{1/2})$ and can be bounded from above by $\eps^{-2}$. Thus, by \eqref{eq:lemTightness2} and Markov's inequality, 
	$$
	\lim_{\rho\searrow 0}\lim\limits_{n\to\infty}\pr(\sup_{|\lambda_1-\lambda_2|\leq \rho} | \tilde  G_n(\lambda_1)- \tilde 
	G_n(\lambda_2)|>\eps) \leq 16\tfrac{K^4}{\eps^4}\eta^2,
	$$ 
	for any $\eta>0$, which proves \eqref{hol68} and completes the proof of  the lemma. 
\end{proof}

\subsection{Proof of the statements in Remark  \ref{rem:gn}} \label{sec:rem}
 
Part (i) is  obvious. 
Part (ii) of  the statement follows  with $g(\mu)=c$  and $\omega \equiv 0$.
	For a proof of part (iii) note that	
	\begin{align*}
	&\lambda \bigg(\hat{g}_n(\lambda)-\frac{1}{t_1-t_0}\int_{t_0}^{t_1}\mu(t)\diff t\bigg) \\
	&= \frac{1}{(t_1-t_0)}\bigg\{\frac{1}{n}\sum_{i=1}^{\lfloor \lambda n\rfloor}\eps_{T_i,n}\id(t_0\leq T_i/n\leq t_1) + \frac{1}{n}\sum_{i=1}^{\lfloor \lambda n\rfloor}\mu\big(\tfrac{T_i}{n}\big)\id(t_0\leq T_i/n\leq t_1)- \lambda\int_{t_0}^{t_1} \mu(t)\diff t\bigg\}\\
	&= \frac{1}{(t_1-t_0)n}\sum_{i=1}^{\lfloor \lambda n\rfloor}\eps_{T_i,n}\id(t_0\leq T_i/n\leq t_1) + \Oc(b_n/n)~.
	\end{align*}
Consequently Assumption \ref{assump:gn}  holds with $\omega(x)=(t_1-t_0)^{-1} \id(t_0\leq x\leq t_1)$.

\medskip

Finally, for a proof of part (iv) note that it follows
from  Lemma \ref{thm:GaussianApprox} and \ref{thm:GaussianApproxRem}   that
	\begin{align*}
	\begin{split}
		\lambda \sqrt{n} \big(\hat{g}_n(\lambda)-g(\mu)\big)
            	&= g\Big(\lambda\sqrt{n}\big(\tilde{\mu}_{h_n}(\lambda,\cdot)-\mu\big)\Big)\\
            	&=\frac{1}{\sqrt{n}} \sum_{i=1}^{\lfloor \lambda n\rfloor} \eps_{T_i,n} g\big(\omega'(T_i/n, \cdot)/h_n\big)+o_\pr(1),
        \end{split}\end{align*}
        by linearity of $g$, where the function $\omega' $ is defined by
        \begin{multline*}
        	\omega'(T_i/n,t) = K_{h_n}^*\big(\tfrac{T_i}{n}-t\big)\id_{[h_n,1-h_n]}(t)\\+\Big\{2c_{\frac{h_n}{\sqrt{2}},i}(t)K_{\frac{h_n}{\sqrt{2}}}\big(\tfrac{T_i}{n}-t\big)-c_{h_n,i}(t)K_{h_n}\big(\tfrac{T_i}{n}-t\big)\Big\} \id_{[0,h_n)\cup(1-h_n,1]}(t).
	\end{multline*}
        Note    $K_{h_n}$ and $K_{h_n}^*$ are Lipschitz continuous with constant $C_k/{h_n}$ where $C_K$ is the Lipschitz constant of $K$ and $K^*$. In particular, if $\supp(K)\subset[-1,1]$, it holds
        \begin{align}\label{eq:example_assump_gn}\begin{split}
	        \Big|\omega'\Big(\tfrac{jb_n}{n},t\Big)-\omega'\Big(\tfrac{jb_n+r}{n},t\Big) \Big|
        	&\leq C_k\frac{r}{nh_n}\id\Big(t\in[\tfrac{jb_n}{n}-h_n,\tfrac{jb_n}{n}+h_n]\cup[\tfrac{jb_n+r}{n}-h_n,\tfrac{jb_n+r}{n}+h_n]\Big) \\
	        &\leq C_k\frac{r}{nh_n}\id\Big(t\in[\tfrac{jb_n}{n}-h_n,\tfrac{(j+1)b_n}{n}+h_n]\Big)\\
	        &= C_k\frac{r}{nh_n}\id\Big((t-h_n)\tfrac{n}{b_n}-1\leq j\leq (t+h_n)\tfrac{n}{b_n}\Big),
	\end{split} \end{align}
        for any $j\in \{1,\dots, \ell_n\}$ and $r\in\{1,\dots,b_n\}$. Since $g$ is bounded,
        the function $\omega_n(x) := g\big(\omega'(x, \cdot)/h_n\big)$ therefore satisfies \eqref{eq:omega}, that is
        $$\sum_{j=1}^{\ell_n}  \sum_{r = 1}^{b_n} \big|\omega\big(\tfrac{jb_n}{n}\big)-\omega\big(\tfrac{r+jb_n}{n}\big)  \big|
        \leq \sum_{j=1}^{\ell_n}  \sum_{r = 1}^{b_n} \frac{\|g\|_{op}}{h_n} \big\|\omega'\big(\tfrac{jb_n}{n},\cdot\big)-\omega'\big(\tfrac{r+jb_n}{n},\cdot\big)  \big\|_2
        = \Oc(b_nh_n^{-1}).
        $$
        To complete the argument, note that the assumption  $\|\omega_n - \omega \|_4 \to 0$ is only needed in the proof of Theorem \ref{thm:convSN} to establish
 the convergence  in \eqref{eq:Riemann_conv}. However,  with $\omega=h_g$, this argument can now be obtained directly
noting  that  the  continuity of $h_g$  implies for any $j\in \{\lceil 2\ell_n h_n\rceil,\dots, \lfloor \ell_n (1-2h_n)\rfloor\}$
	\begin{align*}
	\omega_n(j/\ell_n)& = \langle h_g, \omega'(j/\ell_n, \cdot)/h_n\rangle
	= \frac{1}{h_n} \int_{I_n} h_g(x) \omega'(j/\ell_n, x) \diff x + o(1)\\
	&= \frac{1}{h_n} \int_{I_n} h_g(x) K_{h_n}^*(j/\ell_n- x) \diff x + o(1)\\
	&= h_g(j/\ell_n) \int_{I_n}  \frac{1}{h_n} K_{h_n}^*(j/\ell_n- x) \diff x + o(1)\\
	&= h_g(j/\ell_n) \int_{-1}^1   K^*(x) \diff x + o(1)
	= h_g(j/\ell_n) + o(1)	.
	\end{align*}
	This gives 
	\begin{align*}
	\frac{1}{\ell_n} \sum_{j=1}^{\ell_n} \sigma^2(j/\ell_n)d_{\omega_n}^2(j/\ell_n)
	&= \frac{1}{\ell_n} \sum_{j=1}^{\ell_n} \sigma^2(j/\ell_n)\Big(f_\tau(j/\ell_n)d(j/\ell_n)+\omega_n(j/\ell_n)\int_0^1 d(x)\tau(\diff x)\Big)^2\\
	&=\frac{1}{\ell_n} \sum_{j=1}^{\ell_n} \sigma^2(j/\ell_n)\Big(f_\tau(j/\ell_n)d(j/\ell_n)+h_g(j/\ell_n)\int_0^1 d(x)\tau(\diff x)\Big)^2+o(1),
	\end{align*}
	which converges to $\|d_\omega \sigma\|_2^2$.

\subsection{Proof of Corollary \ref{cortest}} \label{seca.3}
If  $\|d_\omega\sigma\|_2 >0$ the corollary follows immediately from Theorem \ref{thm:convSN}  or Remark \ref{rem1} since
$$ \pr\bigg( \frac{\hat{d}_{2,n}^2(1)-\Delta^2}{\int_\zeta^1 \lambda |\hat{d}_{2,n}^2(\lambda) -\hat{d}_{2,n}^2(1) |\diff \nu(\lambda)} >  q_{1-\alpha} \bigg)\to \left\{\begin{array}{l}0,~\text{if } d_0<\Delta\\ \alpha,~\text{if } d_0=\Delta\\ 1,~\text{if } d_0>\Delta. \end{array}\right. $$
If   $\|d\sigma\|_2 = 0$ , it follows   $d\equiv 0$ by Assumption \ref{assumption:LS} (3), and in this case the probability to reject the null hypothesis
by the decision rule \eqref{eq:test} converges to $\pr(0>\Delta^2)=0$.

\end{document}